\documentclass{amsart}%
\usepackage{amsmath}
\usepackage{amsfonts}
\usepackage{amssymb}
\usepackage{hyperref}
\usepackage{cite}
\usepackage[pdftex]{graphicx}%
\providecommand{\U}[1]{\protect\rule{.1in}{.1in}}
\newtheorem{theorem}{Theorem}[section]
\newtheorem*{theorem*}{Theorem}

\newtheorem{corollary}[theorem]{Corollary}

\newtheorem{definition}[theorem]{Definition}
\newtheorem{example}[theorem]{Example}

\newtheorem{lemma}[theorem]{Lemma}

\newtheorem{proposition}[theorem]{Proposition}

\newtheorem{notation}[theorem]{Notation}

\newtheorem{remark}[theorem]{Remark}

\numberwithin{equation}{section}

\thanks{\footnotemark {$^\dagger$} This research was supported in part by NSF
Grant DMS-1106270.}
\email{bdriver@math.ucsd.edu}
\thanks{\footnotemark {$^*$} This research was supported in part by NSF Grant
DMS-0739164.}
\email{neldredge@unco.edu}
\thanks{\footnotemark {$^{\dagger\dagger}$} This research was supported in part by NSF
Grants DMS-0907293 and DMS-1255574.}
\email{melcher@virginia.edu}
\begin{document}
\title{Hypoelliptic heat kernels on infinite-dimensional Heisenberg groups}
\author[Driver]{Bruce K.~Driver{$^{\dagger}$}}
\address[Bruce K.~Driver]{Department of Mathematics \\
University of California, San Diego \\
La Jolla, California 92093 USA }
\author[Eldredge]{Nathaniel Eldredge{$^{*}$}}
\address[Nathaniel Eldredge]{Department of Mathematics \\ Cornell
  University \\ Ithaca, New York
  14853 USA \\ and \\ School of Mathematical Sciences\\
University of Northern Colorado \\
Greeley, Colorado 80639 USA }
\author[Melcher]{Tai Melcher{$^{\dagger\dagger}$}}
\address[Tai Melcher]{Department of Mathematics\\
University of Virginia\\
Charlottesville, Virginia 22904 USA}

\begin{abstract}
We study the law of a hypoelliptic Brownian motion on an infinite-dimensional
Heisenberg group based on an abstract Wiener space. We show that the endpoint
distribution, which can be seen as a heat kernel measure, is absolutely
continuous with respect to a certain product of Gaussian and Lebesgue measures,
that the heat kernel is quasi-invariant under translation by the
Cameron--Martin subgroup, and that the Radon--Nikodym derivative is Malliavin smooth.

\end{abstract}
\maketitle
\tableofcontents

\section{Introduction\label{h.sec.1}}

In \cite{A35} (also see \cite{A38,A37}), M.~Gordina and the first
author began studying the properties of elliptic heat kernel measures
on certain infinite-dimensional Heisenberg-like groups. There it
was shown that these heat kernel measures enjoyed many of the
quasi-invariance and other smoothness properties found in
finite-dimensional settings and in commutative abstract Wiener space
examples. Later Baudoin, Gordina, and the third author \cite{Baudoin2013}
proved similar results for hypoelliptic heat kernel measures on
infinite-dimensional Heisenberg-like groups. The quasi-invariance
result proved there relied on detailed curvature-dimension inequalities
first suggested by Baudoin, Bonnefont, and Garofalo
\cite{Baudoin2009,Baudoin2010}.

The aim of the current paper is to revisit the hypoelliptic setting
and to show that pure stochastic calculus techniques may be used
to re-prove and actually strengthen the main results for heat kernel
measures in \cite{A35} and \cite{Baudoin2013}. This is done by
developing a concrete representation for the heat kernel measure,
which allows us to show that, in fact, these measures satisfy a
strong definition of smoothness, as given for example in
\cite{Driver2003}. Typically such smoothness results have been
unavailable in the infinite-dimensional subelliptic context and
alternative interpretations must be made (for example, as smoothness
of all appropriate finite-dimensional projections of the measure;
see \cite{BaudoinTeichmann2005,Malliavin1990a,MattinglyPardoux2006}).
To the authors' knowledge, the smoothness results in the present
paper are the first of their type in the infinite-dimensional
subelliptic setting.

\subsection{Heat kernel measures on finite-dimensional Lie groups}

As motivation, let us
briefly recall the finite-dimensional situation.
Let $G$ be a finite-dimensional simply connected Lie group, and let
$\mathfrak{g}$ be the Lie algebra of $G$, identified with the set of
left-invariant vector fields on $G$. Suppose $\{X_{j}\}_{j=1}^{k}
\subset\mathfrak{g}$ is such that $\mathrm{Lie}(\{X_{j}\}_{j=1}^{k})
=\mathfrak{g}$. Then $\{X_{j}\}_{j=1}^{k}$ satisfies the ``bracket
generating'' hypothesis of H\"ormander's Theorem \cite{hormander67}, which
then asserts that the ``sub-Laplacian'' operator $L := X_{1}^{2} + \dots+
X_{k}^{2}$ is \textbf{hypoelliptic}: if $\phi$ is a distribution such that $L
\phi\in C^{\infty}(G)$, then in fact $\phi\in C^{\infty}(G)$. Similarly, the
fundamental solution or heat kernel $p_{t}(x,y)$ of the heat equation $u_{t} -
\frac{1}{2} Lu = 0$ is $C^{\infty}$. In some examples, one can write down an
explicit integral formula for $p_{t}$ from which its smoothness is apparent.
These formulae have been derived many times in the literature; as a small
sample, we mention \cite{gaveau77, randall, lust-piquard, abgr-intrinsic}.

Furthermore, the heat kernel $p_{t}(x,y)$ is strictly positive for all $x,y
\in G$ and $t > 0$. Stated in a more sophisticated way, the heat kernel
measure $p_{t}(e, \cdot)\,dm$, where $m$ is Haar measure, is quasi-invariant
under the action of $G$ on itself by left or right translation. This strict
positivity is typically not immediately apparent even in the settings where an
explicit formula for the heat kernel is available, as these formulae often
involve oscillatory integrals. (However, see \cite{ccgl-positivity-heisenberg}
for an elementary proof of strict positivity for the heat kernel on the real
three-dimensional Heisenberg group.) More generally, strict positivity can be
shown to hold through deeper means, for instance, by the use of parabolic
Harnack inequalities; see for example \cite{vsc-book} and references therein.
On the other hand, for the formula we derive in this paper (Corollary
\ref{gamma-formula}), which applies to step-two nilpotent Lie groups (of
potentially infinite-dimension), strict positivity is obvious by inspection
(see Corollary \ref{gamma-positive}).

In the finite-dimensional nilpotent setting, one can obtain rather precise
information about the integrability of $p_{t}$ and its derivatives; in
particular, $p_{t}$ has Gaussian decay at infinity, with respect to the
Carnot--Carath\'eodory distance on $G$. This is proved in a general setting in
\cite{vsc-book}; in the special case of H-type groups, sharper estimates were
obtained in \cite{eldredge-precise-estimates}. In the infinite-dimensional
setting of the present paper, we are not able to obtain such precise results,
but we are able to derive a Fernique-type theorem (Theorem
\ref{fernique-theorem}), and that together with the $L^{p}$-integrability of
derivatives (Theorem \ref{h.the.7.12}) point roughly in the same direction.

One may think of $p_{t}$ more probabilistically as the end point distribution
of a Brownian motion on $G$. That is, define a Brownian motion on $G$ to be
the unique process $g_{t}$ starting at the identity of $G$ and solving the
Stratonovich stochastic differential equation
\[
\delta g_{t} = \sum_{i} X_{i}(g_{t}) \,\delta B_{t}^{(i)}%
\]
where $B_{t} = (B_{t}^{(1)}, \dots, B_{t}^{(k)})$ is a standard $k$%
-dimensional Brownian motion. Then $L$ is the generator of the Markov process
$g_{t}$, and $p_{t}$ is the transition density of $g_{t}$. In particular, for
each $t$, the law of $g_{t}$ is absolutely continuous with respect to Haar
measure, and its density $p_{t}(e, \cdot)$ is strictly positive and
$C^{\infty}$. Intuitively, despite being driven by a Brownian motion whose
dimension is in general smaller than that of $G$, the process $g_{t}$ is still
able to wander throughout the group $G$.

The purpose of this paper is to obtain, by fairly elementary methods,
analogous results for a class of infinite-dimensional Heisenberg-like groups
modeled on an abstract Wiener space. A key idea is to replace Haar measure,
which no longer exists on our infinite-dimensional groups, by Gaussian
measures on the relevant abstract Wiener space.

We end this section with some standard notation that we will use throughout
the rest of the paper. If $\left(  \Omega,\mathcal{F},\mu\right)  $ is a
probability space and $X:\Omega\rightarrow\mathbb{R}$ is a random variable we
will denote $\int_{\Omega}Xd\mu$ by either $\mathbb{E}X,$ $\mathbb{E}_{\mu}X$,
or simply by $\mu\left(  X\right) $.

\subsection{Heat kernel measures on infinite-dimensional Heisenberg-like
groups}

Here we recall infinite-dimensional Heisenberg-like groups as first
constructed in \cite{A35}. We also define hypoelliptic Brownian motion and
heat kernel measure on these spaces. Infinite-dimensional Heisenberg-like
groups are central extensions of an abstract Wiener space, for which we record
here a standard definition.

\begin{definition}
\label{h.the.1.1}Let $(W,H,\mu)$ be a real \textbf{abstract Wiener space},
that is, $W$ is a separable Banach space, $H$ is a Hilbert space (known as the
\textbf{Cameron--Martin subspace}) which embeds continuously into $W$, and
$\mu$ is the Gaussian Borel measure on $W$ determined by
\[
\mathbb{E}_{\mu}\left[  e^{i\varphi}\right]  =\exp\left(  -\frac{1}%
{2}\left\Vert \varphi|_{H}\right\Vert _{H^{\ast}}^{2}\right)  \text{ for all
}\varphi\in W^{\ast}.
\]
Thus, the covariance of $\mu$ is determined by the inner product of $H$. We
shall assume for simplicity that $(W,H,\mu)$ is non-degenerate: $H$ is dense
in $W$, and (equivalently) the support of the measure $\mu$ is all of $W$.
\end{definition}

\begin{notation}
\label{Wstar} The adjoint of the continuous dense embedding $H \to W$ is a
continuous dense embedding $W^{\ast}\to H$. We denote by $W_{\ast}$ the image
of $W^{\ast}$ in $H$ under this embedding. Equivalently, $W_{\ast}$ consists
of those $h \in H$ such that the linear functional $\langle\cdot, h \rangle\in
H^{*}$ extends continuously to $W$. Recall, though, that even for $h \in H
\setminus W_{\ast}$, the ``functional'' $W \ni x \mapsto\langle x, h
\rangle_{H}$ makes sense as an element of $L^{2}(W, \mu)$.
\end{notation}

\begin{notation}
\label{mut} For $T > 0$, let $\mu_{T}$ be the dilation of $\mu$ defined by
$\mu_{T}(A) = \mu(T^{-1/2} A)$. ($\mu_{0}$ is then a point mass at 0.)
\end{notation}

For further background on abstract Wiener space and Gaussian measures, see for
example \cite{Kuo75, bogachev-gaussian-book}.

\begin{definition}
\label{h.the.1.2}Let $C$ be a finite-dimensional real Hilbert space $C$ of
dimension $d$, whose inner product $C$ is denoted by \textquotedblleft%
\ $\cdot$ \textquotedblright, and suppose that $\omega: W \times W \to C$ is a
continuous, skew-symmetric bilinear form. As explained in \cite[Proposition
3.14]{A35}, we may make $W\times C$ into a Banach Lie group $G$ using the
group multiplication law
\[
\left(  w_{1},c_{1}\right)  \left(  w_{2},c_{2}\right)  =\left(  w_{1}%
+w_{2},c_{1}+c_{2}+\frac{1}{2}\omega\left(  \omega_{1},w_{2}\right)  \right)
.
\]
We shall call such a group \textbf{Heisenberg-like}, by analogy with the
classical Heisenberg group (see Example \ref{H3-example} below). The identity
of $G$ is $(0,0)$, which we denote as $\mathbf{e}$, and the inverse operation
is given by $(w,c)^{-1} = (-w, -c)$. The subset $G_{CM} := H \times C$ is a
subgroup of $G$, which we call the \textbf{Cameron--Martin subgroup}.

The Lie algebra $\mathfrak{g}$ of $G$ may be identified with $W\times C$
equipped with the Lie bracket defined by
\[
\lbrack\left(  w_{1},c_{1}\right)  ,\left(  w_{2},c_{2}\right)  ]=\left(
0,\omega\left(  w_{1},w_{2}\right)  \right)
\]
for all $w_{1},w_{2}\in W$ and $c_{1},c_{2}\in C$. Then $\mathfrak{g}$ is a
nilpotent Banach Lie algebra of step 2. Under this identification, the
exponential map is the identity. The subset $\mathfrak{g}_{CM} := H \times C$
is a Lie subalgebra of $\mathfrak{g}$, which we call the
\textbf{Cameron--Martin subalgebra}.

Throughout this paper, we shall assume that $\omega$ is surjective, or in
other words that the Lie algebra $\mathfrak{g}$ is generated by its subspace
$W \times\{0\}$. This is the analogue of H\"ormander's bracket generating condition.
\end{definition}

\begin{example}
\label{H3-example} If $W = \mathbb{R}^{2}$, $C = \mathbb{R}$, and $\omega$ is
defined by
\[
\omega((x_{1}, y_{1}), (x_{2}, y_{2})) = x_{1} y_{2} - x_{2} y_{1}%
\]
then the group $G$ constructed above is (isomorphic to) the classical
3-dimensional Heisenberg group $\mathbb{H}_{3}$.
\end{example}

\begin{definition}
\label{def-cylinder} A \textbf{smooth cylinder function} on $G$ is a function
$F : G \to\mathbb{R}$ of the form
\[
F(x,c) = \psi(f_{1}(x), \dots, f_{n}(x), c)
\]
for some $n \ge0$, $f_{1}, \dots, f_{n} \in W^{*}$, and $\psi\in C^{\infty
}_{c}(\mathbb{R}^{n} \times C)$.
\end{definition}

\begin{notation}
\label{H-derivative} Given $h \in H$ and $z \in C$, we can compute the partial
derivative of a smooth cylinder function $F$ in the $(h,z)$ direction as
\[
\partial_{(h,z)} F(x,c) = \sum_{i=1}^{n} \partial_{i} \psi(f_{1}(x), \dots,
f_{n}(x), c) \langle f_{i},h \rangle_{H} + \partial_{z} F(x,c).
\]
Note that in the inner product $\langle f_{i}, h \rangle_{H}$ we are
identifying $f_{i} \in W^{*}$ with its image in $W_{\ast}\subset H$ as
explained in Notation \ref{Wstar}. We observe that $\partial_{(h,z)} F$ is
another smooth cylinder function.
\end{notation}

\begin{notation}
\label{vector-fields} We may identify each $X=(h,z)\in\mathfrak{g}_{CM}$ with
a left-invariant vector field $\tilde{X}$, or first-order differential
operator, acting on the cylinder functions on $G$ via
\[
\tilde{X}F(g)=\partial_{(h,z)}(F\circ L_{g})(\mathbf{e}).
\]
where $L_{g}$ is left translation by $g\in G$, that is, $L_{g}(k)=gk$. A
simple computation \cite[Proposition 3.7]{A35} shows that
\begin{equation}
\tilde{X}F(x,c)=\left(  \partial_{\left(  h,z+\frac{1}{2}\omega(x,h)\right)
}F\right)  (x,c). \label{vector-field-omega}%
\end{equation}

\end{notation}

We may now define a group Brownian motion on $G$ and the associated heat
kernel measure.

\begin{definition}
\label{W-Brownian-def} Let $\{B_{t}\}_{t \ge0}$ be a \textbf{$W$-valued
standard Brownian motion}. That is, $\{B_{t}\}$ is a continuous, adapted,
$W$-valued stochastic process defined on some filtered probability space
$(\Omega, \mathcal{F}, \{\mathcal{F}_{t}\}_{t \ge0}, \mathbb{P})$, such that
for all $0 \le s \le t$, we have that $B_{t} - B_{s}$ is independent of
$\mathcal{F}_{s}$ and has $\mu_{t-s}$ as its distribution, where $\mu_{t}$ is
as in Notation \ref{mut}.
\end{definition}

See \cite{Kuo75} for more information about Brownian motion on abstract Wiener space.

\begin{definition}
A \textbf{hypoelliptic Brownian motion on $G$} is the $G$-valued stochastic
process $\{g_{t}\}_{t\geq0}$ which is the solution to the stochastic
differential equation,%
\[
dg_{t}=(L_{g_{t}})_{\ast}\delta B_{t},\quad g_{0}=e
\]
where $(L_{x})_{\ast}$ denotes the differential of left translation by $x\in
G$. (This stochastic differential equation can be interpreted in either the
sense of It\^{o} or Stratonovich; the solutions in the step two nilpotent
setting are the same.) The solution may be written formally as%
\begin{equation}
g_{t}=\left(  B_{t},\frac{1}{2}\int_{0}^{t}\omega\left(  B_{s},dB_{s}\right)
\right)  . \label{gt-integral}%
\end{equation}
(This formal expression will be made precise in Section \ref{h.sec.4} below.)
For fixed $t>0$, the measure $\nu_{t}=\mathrm{Law}(g_{t})$ will be called the
\textbf{hypoelliptic heat kernel measure on $G$} (at time $t$).
\end{definition}

\begin{example}
\label{H3-example-b}Let us return to Example \ref{H3-example}, where
$G=\mathbb{H}_{3}$ is the classical 3-dimensional Heisenberg group. If we
write the 2-dimensional Brownian motion $B_{t}$ on $W=\mathbb{R}^{2}$ in terms
of its components as $B_{t}=(B_{t}^{(1)},B_{t}^{(2)})$, then we have
\[
g_{t}=\left(  B_{t},\frac{1}{2}\int_{0}^{t}B_{s}^{(1)}\,dB_{s}^{(2)}%
-B_{s}^{(2)}\,dB_{s}^{(1)}\right)
\]
so that $g_{t}$ is just the 2-dimensional Brownian motion $B_{t}$ coupled with
its stochastic L\'{e}vy area process.
\end{example}

Let $\Lambda$ be an orthonormal basis for $H$, and for $F$ a smooth cylinder
function as in Definition \ref{def-cylinder} let
\[
LF:=\sum_{h\in\Lambda}\widetilde{(h,0)}^{2}F,\quad\Delta_{H}F:=\sum
_{h\in\Lambda}\partial_{\left(  h,0\right)  }^{2}F
\]
where $\partial_{(h,0)}$ and $\widetilde{(h,0)}$ are as defined in Notation
\ref{H-derivative} and \ref{vector-fields}, respectively. A computation
analogous to that yielding (\ref{vector-field-omega}) (see \cite[Proposition
3.29]{A35}) shows that
\[
\left(  LF\right)  (x,c)=\left(  \Delta_{H}F\right)  (x,c)+\sum_{h\in\Lambda
}\left(  \left(  \frac{1}{4}\partial_{\left(  0,\omega\left(  x,h\right)
\right)  }^{2}+\partial_{\left(  0,\omega\left(  x,h\right)  \right)
}\partial_{\left(  h,0\right)  }\right)  F\right)  (x,c).
\]
Furthermore from \cite[Corollary 4.5]{A35}, $\nu_{t}$ weakly solves the heat
equation
\[
\partial_{t}\nu_{t}=\frac{1}{2}L\nu_{t}\text{ with }\lim_{t\downarrow0}\nu
_{t}=\delta_{\mathbf{e}},
\]
that is, for all smooth cylinder functions $F$ as in
Definition \ref{def-cylinder},
\[
\nu_{T}\left(  F\right)  =F\left(  \mathbf{e}\right)  +\frac{1}{2}\int_{0}%
^{T}\mathbb{\nu}_{t}\left(  LF\right)  dt,\quad\text{ for all }T>0.
\]

\subsection{Summary of results}

Here we briefly describe the main results proved in this paper. A key
ingredient in our results is the identity of Theorem
\ref{key-identity-inf}, proved in certain finite-dimensional cases by
M.~Yor (see Section \ref{h.sec.2}): if $B_{t}$ is a Brownian motion on $W$ and
$A:H\rightarrow H$ is Hilbert--Schmidt, then for any $T\geq0$ and any bounded
measurable $f:W\rightarrow\mathbb{R}$, we have
\begin{equation}
\mathbb{E}\left[  f\left(  B_{T}\right)  e^{i\int_{0}^{T}\left\langle
AB_{t},dB_{t}\right\rangle _{H}}\right]  =\mathbb{E}\left[  f\left(
B_{T}\right)  e^{-\frac{1}{2}\int_{0}^{T}\left\Vert AB_{t}\right\Vert _{H}%
^{2}dt}\right]  . \label{key-identity-intro}%
\end{equation}
(The process $AB_{t}$ and the stochastic integral on the left side are defined
precisely in Section \ref{h.sec.3}, see Propositions \ref{ABt-bounded} and
\ref{integral-bounded}.) Using the identity (\ref{key-identity-intro}), in
Section \ref{h.sec.4} we derive a formula for the heat kernel $\nu_{T}$. For
$\lambda\in C$, define the Hilbert--Schmidt operator $\Omega_{\lambda
}:H\rightarrow H$ by $\langle\Omega_{\lambda}h,k\rangle_{H}=\omega
(h,k)\cdot\lambda$ for $h,k\in H$, and let $\rho_{T}$ be the random linear
transformation on $C$ defined by
\[
\rho_{T}\lambda\cdot\lambda^{\prime}=\frac{1}{4}\int_{0}^{T}\langle
\Omega_{\lambda}B_{t},\Omega_{\lambda^{\prime}}B_{t}\rangle_{H}\,dt.
\]
In this notation, our Corollary \ref{gamma-formula} states that
\[
\nu_{T}(dx,dc)=\gamma_{T}(x,c)\mu_{T}(dx)m(dc)
\]
where $m$ is Lebesgue measure on $C$ and
\begin{equation}
\gamma_{T}(x,c):=\mathbb{E}\left.  \left[  \frac{\exp\left(  -\frac{1}{2}%
\rho_{T}^{-1}c\cdot c\right)  }{\sqrt{(2\pi)^{d}\det\rho}}\right\vert
B_{T}=x\right]  . \label{gammaT-formula-intro}%
\end{equation}
In particular, $\nu_{T}$ is absolutely continuous with respect to product
measure $\mu_{T}\otimes m$ on $G=W\times C$. To the best of our knowledge, the
formula (\ref{gammaT-formula-intro}) is new even in the finite-dimensional case.

As a first application of (\ref{gammaT-formula-intro}), we prove in Section
\ref{fernique-sec} a Fernique-type theorem (see Proposition
\ref{fernique-theorem}) on the integrability of $\nu_{T}$: there exists an
$\varepsilon>0$ such that
\[
\int e^{\frac{\varepsilon}{T}(\Vert x\Vert_{W}^{2}+|c|_{C})}\,\nu
_{T}(dx,dc)<\infty.
\]
Necessary ingredients include estimates on the integrability of $\rho_{T}%
^{-1}$, which are the subject of Section \ref{rhoT-sec}.

By further analysis of the formula (\ref{gammaT-formula-intro}), and use of
the estimates of Section \ref{rhoT-sec}, we show in Sections \ref{h.sec.6} and
\ref{h.sub.7.1} that the heat kernel is quasi-invariant under translations by
$G_{CM}$ and is infinitely differentiable in those directions. In particular,
for $\tilde{X}$ the left-invariant vector field on $G$ associated to
$X\in\mathfrak{g}_{CM}$,
\[
\tilde{X}^{\ast}=-\tilde{X}+\psi_{T}^{X}%
\]
where $\psi_{T}^{X}:G\rightarrow\mathbb{R}$, essentially the first logarithmic
derivative of $\nu_{T}$, is a \textquotedblleft Malliavin
smooth\textquotedblright\ function in the sense that $\psi_{T}^{X}$ is in
$L^{\infty-}(\nu_{T})=\cap_{p\in\lbrack1,\infty)}L^{p}(\nu_{T})$ and all of
its derivatives exist and are in also in $L^{\infty-}(\nu_{T})$; see Lemma
\ref{h.the.7.8} and Corollaries \ref{h.the.7.12-2} and \ref{h.the.7.12}.
Similar techniques were used in \cite{Dobbs2013} to study elliptic heat kernel
measures on $G$.

\subsection{Measurable group actions\label{s.mga}}

It was first observed in \cite{A35} that the assumption that $\omega$ be a
continuous operator on $W\times W$ is not strictly necessary. A standard fact
about abstract Wiener spaces is that, if $K$ is any Hilbert space and $T : W
\to K$ is a continuous linear operator, then the restriction $T|_{H} : H \to
K$ is Hilbert--Schmidt \cite[Corollary 1.4.4]{Kuo75}. It follows that, if
$\omega: W \times W \to C$ is continuous, then its restriction $\omega|_{H
\times H} : H \times H \to C$ is Hilbert--Schmidt, in the sense that for any
orthonormal basis $\Lambda$ of $H$, we have $\|\omega\|_{HS}^{2} := \sum_{h,k
\in\Lambda} |\omega(h,k)|_{C}^{2} < \infty$.

Conversely, suppose that we are merely given a skew-symmetric bilinear form
$\omega:H\times H\rightarrow C$ which is Hilbert--Schmidt. We will show in
Section \ref{h.sec.3} that the formula (\ref{gt-integral}) still makes sense,
\emph{even if $\omega$ does not extend continuously to $W\times W$}. Indeed,
the results of this paper (which are all expressed in terms of the process
$g_{t}$) will be proved under this weaker assumption. In this case, there is
no canonical way to extend $\omega$ to $W\times W$, and thus it does not
really make sense to speak of $G$ as a group. However, the Cameron--Martin
group $G_{CM}$ is still perfectly well-defined, and we obtain a measurable
left and right action of $G_{CM}$ on the measurable space $G$. In particular,
fix $h\in H$ and let $T_{h}:H\rightarrow C$ be the linear map given by
$T_{h}k:=\omega(h,k)$ for any $k\in H$. Then, for any orthonormal basis
$\{e_{j}\}_{j=1}^{d}$ of $C$, we have that the linear functional $\ell
_{j}:H\rightarrow\mathbb{R}$ defined by $\ell_{j}(k):=T_{h}k\cdot e_{j}%
=\omega(h,k)\cdot e_{j}$ extends to a $\mu$-measurable linear functional
$\hat{\ell}_{j}$ on $W$; see for example \cite[Theorem 2.10.11]%
{bogachev-gaussian-book}. Thus, we may define
\begin{equation}
\omega(h,w):=\sum_{j=1}^{d}\hat{\ell}_{j}(w)e_{j}, \label{e.ext}%
\end{equation}
and hence
\begin{equation}
(w,c)\mapsto(h,z)(w,c):=\left(  h+w,z+c+\frac{1}{2}\omega(h,w)\right)
\label{e.r}%
\end{equation}
is a measurable transformation on $G$. In fact, under the assumption that
$\omega$ is Hilbert--Schmidt, this does not depend on $C$ being
finite-dimensional. The adjoint $T_{h}^{\ast}:C\rightarrow H$ is also
Hilbert--Schmidt, and we may write $\ell_{j}(k)=\langle k,T_{h}^{\ast}%
e_{j}\rangle_{H}$. Then $\Vert\hat{\ell}_{j}\Vert_{L^{2}(\mu)}=\Vert
T_{h}^{\ast}e_{j}\Vert_{H}$, which implies that
\[
\sum_{j=1}^{\infty}\Vert\hat{\ell}_{j}\Vert_{L^{2}(\mu)}^{2}<\infty,
\]
and so $\sum_{j=1}^{\infty}\hat{\ell}_{j}(w)^{2}<\infty$ $\mu$-a.s and
equation (\ref{e.ext}) makes sense with $d=\infty$. Thus, for example, the
quasi-invariance results of Section \ref{h.sec.6} can be interpreted as
statements about how the measure $\nu_{t}$ on $G$ behaves under the left
action by elements $(h,z)\in G_{CM}$ as in (\ref{e.r}), and the analogously
defined right action. Further discussion can be found in \cite[Section 9]{A35}.

For concreteness, however, we encourage the reader to continue to think of the
case when $\omega$ does have a continuous extension to $W\times W$, in which
$G$ is an honest group.

\section{Quadratic Brownian integrals in finite dimensions}\label{h.sec.2}

This section is devoted to the discussion of the identity
(\ref{key-identity-intro}) in the case where $W$ is finite-dimensional. 

\begin{theorem}
\label{h.the.2.1}Let $\left\{  B_{t}\right\}  _{t\geq0}$ be an $N$-dimensional
Brownian motion, $A$ be an $N\times N$ skew-symmetric matrix, and $T>0$. Then,
for any measurable $f:\mathbb{R}^{N}\rightarrow\mathbb{C}$ such that
$\mathbb{E}\left\vert f\left(  B_{T}\right)  \right\vert <\infty,$
\begin{equation}
\mathbb{E}\left[  f\left(  B_{T}\right)  e^{i\int_{0}^{T}AB_{t}\cdot dB_{t}%
}\right]  =\mathbb{E}\left[  f\left(  B_{T}\right)  e^{-\frac{1}{2}\int%
_{0}^{T}\left\vert AB_{t}\right\vert ^{2}dt}\right]  . \label{h.equ.2.1}%
\end{equation}

\end{theorem}

\begin{remark}
  This result was proved in the case of the 3-dimensional Heisenberg
  group $\mathbb{H}_3$ (in which $N=2$, $A = \begin{pmatrix} 0 & 1
    \\ -1 & 0 \end{pmatrix}$, and $\int_0^T AB_t\cdot dB_t$ is the L\'evy area process) by
  M.~Yor \cite{yor-remarques-levy}.  An alternative derivation can be
  found in \cite[Section 2.1]{mansuy-yor}.  The relationship between
  the integrals $\int_0^T AB_t\cdot dB_t$ and $\int_0^T |A B_t|^2\,dt$ was
  studied by P.~L\'evy \cite{levy40, levy50} but to the best of our
  understanding the identity (\ref{h.equ.2.1}) is not contained in his
  work.  A similar computation appears in \cite{kuelbs-li-2005} which
  effectively obtains (\ref{h.equ.2.1}) in the case $f=1$.   Here
we provide a proof based on Yor's result for the stochastic L\'{e}vy area.  
In Appendix \ref{appendix},
we also provide another self-contained proof of Theorem \ref{h.the.2.1},
based on analysis of the infinitesimal generator of $g_t$.

\end{remark}

\noindent
{\it Proof of Theorem \ref{h.the.2.1}.}
  We include here a sketch of Yor's argument for $\mathbb{H}_3$; we
   will then show how the general case follows.  Suppose $N=2$ and $A = 
   \begin{pmatrix}
     0 & a \\ -a & 0
   \end{pmatrix}
   $
   for some $a \in \mathbb{R}$.  By the rotational invariance of the
   Brownian motion $B$ and the fact that $A$ commutes with rotations,
   it is sufficient to establish (\ref{h.equ.2.1}) with $f(B_T)$
   replaced by $g(|B_T|)$; that is, to show that
   \begin{equation}\label{abs-BT-cond}
     \mathbb{E} \left[ \left. e^{i \int_0^T A B_t\cdot dB_t} \right\vert |B_T| \right]
     = \mathbb{E} \left[ \left. e^{-\frac{1}{2} \int_0^T |A B_t|^2\,dt} \right\vert
       |B_T| \right].
  \end{equation}
  Now we observe that
  \begin{equation*}
    \beta_t := \int_0^t \frac{B_s}{|B_s|} \cdot dB_s, \quad \gamma_t :=
    \int_0^t \frac{A B_s}{|A B_s|} \cdot dB_s
  \end{equation*}
  are two independent one-dimensional Brownian motions.  (The
  integrals are well-defined because, almost surely, $B_t \ne 0$ for
  almost every $t$.  They can be seen to be Brownian motions by
  L\'evy's characterization; each is a continuous martingale with
  quadratic variation $t$.  Finally, since $A$ is skew-symmetric, it
  follows from It\^o's isometry that $\beta, \gamma$ are
  uncorrelated and hence independent.)  We note that since $|A B_s| =
  a |B_s|$, we can write 
  \begin{equation}\label{A-gamma}
    Z_t = \int_0^t A B_s \cdot dB_s = a \int_0^t |B_s|\,d\gamma_s.
  \end{equation}
  Now if we let $S_t = |B_t|^2$,
  by It\^o's formula we have
  \begin{equation*}
    S_t = 2 \int_0^t  B_s \cdot dB_s + 2t = 2 \int_0^t \sqrt{S_t} \,d\beta_t + 2 t.
  \end{equation*}
  In particular, $S$ is $\sigma(\beta)$-measurable, and hence $S$ (and
  also $|B| = \sqrt{S}$) is independent of the process $\gamma$.  (For
  details, see \cite{yor-filtrations} and references therein.)  Thus,
  (\ref{A-gamma}) implies that, conditioned on $|B|$, $Z_T$ is
  Gaussian with variance given by $a^2 \int_0^T |B_t|^2\,dt = \int_0^T
  |A B_t|^2\,dt$.  Hence by the Gaussian Fourier transform, we have
  \begin{equation*}
    \mathbb{E} \left[ \left. e^{i Z_T} \right\vert |B|\right] = e^{-\frac{1}{2}
      \int_0^T |A B_t|^2\,dt}.
  \end{equation*}
  Conditioning on $|B_T|$ we have (\ref{abs-BT-cond}).

  For arbitrary $N$, we begin with the case that $A$ is
  quasi-diagonal, that is, block diagonal with its nonzero blocks of
  the form $
  \begin{pmatrix}
    0 & a_i \\ -a_i & 0
  \end{pmatrix}
  $.  If $f$ is of the form $f(x_1, \dots, x_N) = f_1(x_1) \dots
  f_N(x_N)$ with $f_1, \dots, f_n$ bounded and measurable, then (\ref{h.equ.2.1}) follows immediately from the $N=2$
  case by using independence.  Then the case of general bounded
  measurable $f$ follows from the multiplicative system theorem
  \cite[Appendix A, p.~309]{Janson1997}, and for general $f$ we
  can use a truncation argument.

  Finally, an arbitrary skew-symmetric $A$ can be written $A = U
  \tilde{A} U$, where $U$ is orthogonal and $\tilde{A}$ is
  quasi-diagonal.  ($U$ can be taken to have rows given by the real
  and imaginary parts of the eigenvectors of the Hermitian complex
  matrix $iA$.)  We have shown that (\ref{h.equ.2.1}) holds for
  $\tilde{A}$; if we replace $B_t$ by the Brownian motion $U B_t$ and
  $f$ by $f \circ U^*$, we have the desired result for $A$.
\hfill$\square$

\section{Quadratic Brownian integrals in infinite dimensions\label{h.sec.3}}

It is now fairly easy to generalize Theorem \ref{h.the.2.1} to the
infinite-dimensional setting using a finite-dimensional approximation
argument. Before doing this we need to construct the infinite-dimensional
stochastic processes involved. We give a self-contained construction that
suffices for our purposes, but for a more general view of Hilbert space
stochastic calculus, see \cite{Metivier82}. The It\^{o} integral relative to
Brownian motion in an abstract Wiener space, is discussed in Kuo \cite[pages
188-207, especially Theorem 5.1]{Kuo75}, \cite[p. 5]{KS1991}, and the appendix
in \cite{Driver1997b}.

\begin{notation}
Let $HS=HS(H)$ denote the Hilbert space of Hilbert--Schmidt operators on $H$,
with the usual norm $\Vert A\Vert_{HS}^{2}:=\sum_{h\in\Lambda}\Vert
Ah_{i}\Vert_{H}^{2}$, where $\Lambda$ is any orthonormal basis for $H$. Let
$HS_{0}=HS_{0}(H)$ denote the subspace consisting of those operators which
extend continuously to $W$, and whose range is finite-dimensional and
contained in $W_{\ast}$.
\end{notation}

\begin{lemma}
\label{l.HS0-form}A Hilbert--Schmidt operator $A$ is in $HS_{0}$ iff there
exists a finite rank orthogonal projection $P:H\rightarrow H$ such that
$A=PAP,$ $\operatorname*{Ran}\left(  P\right)  \subset W_{\ast},$ and $P$
extends continuously to $W.$
\end{lemma}

\begin{proof}
If $A=PAP$ with $P$ being a projection as in the statement of the lemma, then
it is clear that $A\in HS_{0}.$ Conversely, if $A\in HS_{0},$ let $\left\{
u_{i}\right\}  _{i=1}^{n}\subset W_{\ast}$ be an orthonormal basis for
$\operatorname*{Ran}\left(  A\right)  .$ Then for all $h\in H$
\[
Ah=\sum_{i=1}^{n}\left\langle Ah,u_{i}\right\rangle u_{i}=\sum_{i=1}%
^{n}\left\langle h,A^{\ast}u_{i}\right\rangle u_{i}=\sum_{i=1}^{n}\left\langle
h,v_{i}\right\rangle u_{i}
\]
where $v_{i}:=A^{\ast}u_{i}\in W_{\ast}$ for all $i.$ We now define $P$ to be
orthogonal projection onto
\[
\operatorname*{span}\left\{  u_{i},v_{i}:1\leq i\leq n\right\}
=\operatorname*{Ran}\left(  A\right)  +\operatorname*{Ran}\left(  A^{\ast
}\right)  .
\]
It is now a simple matter to check that $AP=A=PA$ on $H$ and $A=PAP$ on $W.$
\end{proof}

\begin{lemma}
\label{HS-approx}If $A\in HS$ and $\left\{  S_{n}\right\}  _{n=1}^{\infty}$
and $\left\{  T_{n}\right\}  _{n=1}^{\infty}$ are bounded operators on $H$
such that $S_{n}\overset{s}{\rightarrow}I$ and $T_{n}\overset{s}{\rightarrow
}I$ (strong convergence), then $\left\Vert S_{n}AT_{n}^{\ast}-A\right\Vert
_{HS}\rightarrow0$ as $n\rightarrow\infty.$ [Note that $T_{n}%
\overset{s}{\rightarrow}I$ does not necessarily imply $T_{n}^{\ast
}\overset{s}{\rightarrow}I.]$
\end{lemma}

We would like to thank Mart\'in Argerami \cite{mse-hilbert-schmidt} for
suggesting the following proof.

\begin{proof}
Let $s_{n}:=S_{n}-I$ and $t_{n}:=T_{n}-I$ so that $S_{n}=I+s_{n}$ and
$T_{n}=I+t_{n}$ with $s_{n},t_{n}\overset{s}{\rightarrow}0$ as $n\rightarrow
\infty.$ By the uniform boundedness principle we know $C:=\sup_{n}\left(
\left\Vert S_{n}\right\Vert _{op}\vee\left\Vert T_{n}\right\Vert _{op}\right)
<\infty.$ Then%
\begin{align*}
\left\Vert S_{n}AT_{n}^{\ast}-A\right\Vert _{HS }  &  =\left\Vert
S_{n}A\left(  I+t_{n}^{\ast}\right)  -A\right\Vert _{HS }\\
&  =\left\Vert s_{n}A+S_{n}At_{n}^{\ast}\right\Vert _{HS }\\
&  \leq\left\Vert s_{n}A\right\Vert _{HS }+C\left\Vert At_{n}^{\ast
}\right\Vert _{HS }=\left\Vert s_{n}A\right\Vert _{HS}C\left\Vert t_{n}%
A^{\ast}\right\Vert _{HS }.
\end{align*}
By the dominated convergence theorem we have, for any orthonormal basis
$\Lambda$,
\[
\lim_{n\rightarrow\infty}\left\Vert s_{n}A\right\Vert _{HS}^{2} =\lim
_{n\rightarrow\infty}\sum_{h\in\Lambda}\left\Vert s_{n}Ah\right\Vert _{H}%
^{2}=\sum_{h\in\Lambda}\lim_{n\rightarrow\infty}\left\Vert s_{n}Ah\right\Vert
_{H}^{2}=\sum_{h\in\Lambda}0=0
\]
and similarly $\lim_{n\rightarrow\infty}\left\Vert t_{n}A^{\ast}\right\Vert
_{HS }^{2}=0.$
\end{proof}

\begin{corollary}
\label{HS0-dense} $HS_{0}$ is dense in $HS$.
\end{corollary}

\begin{proof}
Since $W_{\ast}$ is dense in $H$, we can choose an orthonormal basis
$\{h_{i}\}_{i=1}^{\infty}$ for $H$ with $\Lambda\subset W_{\ast}$. Set $P_{n}
h = \sum_{i=1}^{n} \langle h, h_{i} \rangle_{H} h_{i}$ to be the orthogonal
projection onto the span of $\{h_{1}, \dots, h_{n}\}$. Note that $P_{n}$ is
self-adjoint and $P_{n} \to I$ strongly. Then for any $A \in HS$, it is simple
to verify that $P_{n} A P_{n} \in HS_{0}$, and taking $S_{n} = T_{n} =
T_{n}^{*} = P_{n}$ in Lemma \ref{HS-approx}, we have $P_{n} A P_{n} \to A$ in
$HS$-norm.
\end{proof}

\begin{notation}
Let $(\Omega, \mathcal{F}, \{\mathcal{F}_{t}\}_{t \ge0}, \mathbb{P})$ be a
filtered probability space on which there is defined a $W$-valued standard
Brownian motion $\{B_{t}\}_{t \ge0}$ as in Definition \ref{W-Brownian-def}.
Fix $T > 0$ and let $\mathcal{M}_{T}(H)$ denote the vector space of
continuous, square-integrable, $H$-valued martingales up to time $T$ defined
on $\Omega$. We equip $\mathcal{M}_{T}(H)$ with the Banach norm
\[
\|M\|_{\mathcal{M}_{T}(H)}^{2} := \mathbb{E} \left[  \sup_{t \in[0,T]}
\|M_{t}\|_{H}^{2}\right]  .
\]
We likewise define $\mathcal{M}_{T} = \mathcal{M}_{T}(\mathbb{R})$ as the
space of scalar-valued continuous square-integrable martingales with the
analogous Banach norm.
\end{notation}

\begin{proposition}
\label{ABt-bounded} The linear map
\[
HS_{0} \ni A \mapsto\{A B_{t}\}_{0 \le t \le T} \in\mathcal{M}_{T}(H)
\]
is bounded. Hence it extends continuously to a map $HS \to\mathcal{M}_{T}(H)$,
which we will still denote $A B_{t}$.
\end{proposition}

\begin{proof}
For $A\in HS_{0}$, there exists by Lemma \ref{l.HS0-form} a finite rank
projection $P$ with $\operatorname*{Ran}\left(  P\right)  \subset W_{\ast}$
such that $A=PAP.$ Then $b_{t}:=PB_{t}$ is a standard $\operatorname*{Ran}%
\left(  P\right)  $-valued Brownian motion and therefore $M_{t}:=AB_{t}%
=PAPB_{t}=PAb_{t}$ is a $\operatorname*{Ran}\left(  P\right)  $-valued
continuous martingale. Now $\{\Vert M_{t}\Vert_{H}\}_{t\geq0}$ is a continuous
submartingale and so Doob's maximal inequality gives
\[
\Vert M\Vert_{\mathcal{M}_{T}(H)}^{2}\leq4\mathbb{E}\Vert M_{T}\Vert_{H}^{2}.
\]
A simple computation shows that $\mathbb{E}\Vert M_{T}\Vert_{H}^{2}%
=\mathbb{E}\Vert Ab_{T}\Vert_{H}^{2}=T\Vert A\Vert_{HS}^{2}$, which completes
the proof.
\end{proof}

\begin{remark}
One can replace the space $\mathcal{M}_{T}(H)$ of square-integrable
martingales with spaces $\mathcal{M}_{T}^{p}(H)$ of $L^{p}$ martingales, $1
\le p < \infty$, with an analogous norm. The corresponding statements still
hold, showing for instance that $A B_{t}$ is an $L^{p}$ martingale, upon
replacing Doob's maximal inequality with the Burkholder--Davis--Gundy inequality.
\end{remark}

For $A=PAP\in HS_{0}$ as in Lemma \ref{l.HS0-form}, we may interpret,%
\[
\int_{0}^{t}\langle AB_{s},dB_{s}\rangle_{H}=\int_{0}^{t}\langle
PAPB_{s},dB_{s}\rangle_{H}=\int_{0}^{t}\langle APB_{s},dPB_{s}\rangle_{H}%
=\int_{0}^{t}\langle Ab_{s},db_{s}\rangle_{H}%
\]
where $b_{t}:=PB_{t}$ is a standard $\operatorname*{Ran}\left(  P\right)
$-valued Brownian motion. Hence we are dealing solely with finite-dimensional
stochastic calculus.

\begin{proposition}
\label{integral-bounded} The linear map
\[
HS_{0} \ni A \mapsto\int_{0}^{\cdot}\langle A B_{s}, dB_{s} \rangle_{H}
\in\mathcal{M}_{T}%
\]
is bounded. Hence it extends continuously to a map $HS \to\mathcal{M}_{T}$
which we shall still denote by $\int_{0}^{\cdot}\langle A B_{s}, dB_{s}
\rangle_{H}$.
\end{proposition}

In particular, for fixed $T$, we have a continuous linear map $HS \ni A
\mapsto\int_{0}^{T} \langle A B_{s}, dB_{s} \rangle\in L^{2}(\mathbb{P})$.

\begin{proof}
Let $A\in HS_{0}$ and set $M_{t}=\int_{0}^{t}\langle AB_{s},dB_{s}\rangle
_{H}=\int_{0}^{t}\langle Ab_{s},db_{s}\rangle_{H}$ as above, where
$b_{t}=PB_{t}$ is a standard $\operatorname*{Ran}\left(  P\right)  $-valued
Brownian motion. By Doob's maximal inequality and It\^{o}'s isometry, we have
\[
\Vert M\Vert_{\mathcal{M}_{T}}^{2}\leq4\mathbb{E}|M_{T}|^{2}=4\int_{0}%
^{T}\mathbb{E}\Vert Ab_{t}\Vert_{H}^{2}\,dt=2T\Vert A\Vert_{HS}^{2}%
\]
which completes the proof.
\end{proof}

Now a simple limiting argument shows that Theorem \ref{h.the.2.1} still holds
in this infinite-dimensional setting.

\begin{theorem}
\label{key-identity-inf} Let $A \in HS$ be skew-adjoint (i.e. $A^{*} = -A$)
and $T>0$. Then, for any bounded measurable $f : W \to\mathbb{R}$,
\begin{equation}
\label{key-identity-eqn}\mathbb{E}\left[  f\left(  B_{T}\right)  e^{i\int%
_{0}^{T}\left\langle AB_{t},dB_{t}\right\rangle _{H}}\right]  =\mathbb{E}%
\left[  f\left(  B_{T}\right)  e^{-\frac{1}{2}\int_{0}^{T}\left\Vert
AB_{t}\right\Vert _{H}^{2}dt}\right]  .
\end{equation}

\end{theorem}

\begin{proof}
Suppose first that $f$ is a bounded cylinder function, that is, $f(x) =
\psi(\langle h_{1}, x \rangle_{H}, \dots, \langle h_{N}, x \rangle_{H})$ for
some $h_{1}, \dots, H_{N} \in W_{\ast}$. Extend $\{h_{1}, \dots, h_{N}\}$ to
an orthonormal basis $\{h_{n}\} \subset W_{\ast}$ for $H$, and use this basis
to define $P_{n}$ as in Corollary \ref{HS0-dense}. In particular, $P_{n} A
P_{n} \in HS_{0}$ (and is also skew-adjoint), and $P_{n} A P_{n} \to A$ in
$HS$ norm. Now by Theorem \ref{h.the.2.1}, we have
\[
\mathbb{E}\left[  f\left(  P_{n}B_{T}\right)  e^{i\int_{0}^{T}\left\langle
P_{n}AP_{n}B_{t},dB_{t}\right\rangle _{H}}\right]  =\mathbb{E}\left[  f\left(
P_{n}B_{T}\right)  e^{-\frac{1}{2}\int_{0}^{T}\left\Vert P_{n} AP_{n}%
B_{t}\right\Vert _{H}^{2}dt}\right]  .
\]
Now we pass to the limit. For $n \ge N$ we have $f(P_{n} B_{T}) = f(B_{T})$.
Next, the continuity of the map in Proposition \ref{integral-bounded} shows
that $\int_{0}^{T}\left\langle P_{n}AP_{n}B_{t},dB_{t}\right\rangle _{H}
\to\int_{0}^{T}\left\langle AB_{t},dB_{t}\right\rangle _{H}$ in $L^{2}%
(\mathbb{P})$.

Finally, Proposition \ref{ABt-bounded} tells us that $P_{n} A P_{n} B_{t}$
converges to $A B_{t}$ in $\mathcal{M}_{T}(H)$; that is, as random elements of
$C([0,T]; H)$, they converge in $L^{2}(\mathbb{P})$. The map $\mathbf{x}
\mapsto\int_{0}^{T} \|\mathbf{x}(t)\|^{2}_{H}\,dt$ is continuous on
$C([0,T];H)$, so by continuous mapping we have $\int_{0}^{T} \|P_{n} A P_{n}
B_{t}\|_{H}^{2}\,dt \to\int_{0}^{T} \|A B_{t}\|_{H}^{2}\,dt$ in probability.

Putting this all together and using continuous mapping again, we have
\begin{align*}
f\left(  P_{n}B_{T}\right)  e^{i\int_{0}^{T}\left\langle P_{n}AP_{n}%
B_{t},dB_{t}\right\rangle _{H}}  &  \rightarrow f\left(  B_{T}\right)
e^{i\int_{0}^{T}\left\langle AB_{t},dB_{t}\right\rangle _{H}} &  &
\text{i.p.}\\
f\left(  P_{n}B_{T}\right)  e^{-\frac{1}{2}\int_{0}^{T}\left\Vert P_{n}%
AP_{n}B_{t}\right\Vert _{H}^{2}dt}  &  \rightarrow f\left(  B_{T}\right)
e^{-\frac{1}{2}\int_{0}^{T}\left\Vert AB_{t}\right\Vert _{H}^{2}dt} &  &
\text{i.p.}%
\end{align*}
Everything in sight is bounded, so the dominated convergence theorem gives us
the conclusion, still assuming that $f$ is a cylinder function. An application
of the multiplicative system theorem 
then covers the case that $f$ is merely bounded and measurable.
\end{proof}

\section{Heisenberg heat kernels}

\label{h.sec.4}

As in previous sections, $(W,H,\mu)$ is an abstract Wiener space, $B_{t}$ is a
Brownian motion on $W$, $C$ is a finite-dimensional Hilbert with inner product
$\cdot$, and $\omega: H \times H \to C$ is a skew-symmetric Hilbert--Schmidt
bilinear form which is surjective.

\begin{notation}
For $\lambda\in C$, let $\Omega_{\lambda}\in HS$ be the Hilbert--Schmidt
operator defined by $\langle\Omega_{\lambda}h, k \rangle= \omega(h,k)
\cdot\lambda$ for all $h,k\in H$.
\end{notation}

\begin{definition}
\label{h.the.4.1} A \textbf{hypoelliptic Brownian motion} on $G$ is the
$G$-valued process $g_{t} = (B_{t}, Z_{t})$, where $Z_{t} = \frac{1}{2}
\int_{0}^{t} \omega(B_{s}, dB_{s})$. To be precise, $Z_{t}$ is defined by
$Z_{t} \cdot\lambda= \frac{1}{2} \int_{0}^{t} \langle\Omega_{\lambda}B_{t},
dB_{t} \rangle_{H}$, where the stochastic integral is defined as in
Proposition \ref{integral-bounded}. For $T>0$, let $\nu_{T}=\mathrm{Law}%
(g_{T})$ denote the \textbf{hypoelliptic heat kernel measure} at time $T$ on
$G$.
\end{definition}

We note the scaling relation
\begin{equation}
g_{ct}\overset{d}{=}(\sqrt{c}B_{t},cZ_{t})\quad\text{in law}. \label{gt-scale}%
\end{equation}

Alternatively, following the development in Section \ref{h.sec.3}, we could
also define $Z_{t}$ as the limit of the continuous $C$-valued processes
$\frac{1}{2} \int_{0}^{t} \omega(P_{n} B_{s}, dP_{n} B_{s})$, for $P_{n}$ as
in Corollary \ref{HS0-dense}.

\begin{notation}
\label{End-notation} Let $\operatorname*{End}(C)$ denote the space of linear
transformations of the finite-dimensional Hilbert space $C$. We will consider
$\operatorname*{End}(C)$ as a finite-dimensional Banach space equipped with
the operator norm, which we denote by $\|\cdot\|_{op}$. Also, let
$\operatorname*{End}_{+}(C) \subset\operatorname*{End}(C)$ denote the closed
cone of self-adjoint, nonnegative definite transformations.
\end{notation}

\begin{notation}
\label{rhoT-def} For $\mathbf{x},\mathbf{y}\in C([0,T];H)$, let $\rho
_{T}(\mathbf{x},\mathbf{y}) \in\operatorname*{End}(C)$ be defined by
\[
\rho_{T}(\mathbf{x},\mathbf{y})\lambda\cdot\lambda^{\prime}:=\frac{1}{4}%
\int_{0}^{T}\langle\Omega_{\lambda}\mathbf{x}(s),\Omega_{\lambda^{\prime}%
}\mathbf{y}(s)\rangle_{H}\,ds.
\]
As usual, we will let $\rho_{T}(\mathbf{x})=\rho_{T}(\mathbf{x},\mathbf{x})$,
and note that $\rho_{T}(\mathbf{x}) \in\operatorname*{End}_{+}(C)$.

By Proposition \ref{ABt-bounded}, $\rho_{T}(\mathbf{x},\mathbf{y})$ also makes
sense (as a random linear transformation) if one or both of $\mathbf{x}%
,\mathbf{y}$ is replaced by a $W$-valued Brownian motion $B$. Henceforth
$\rho_{T}$ by itself will denote the random linear transformation $\rho
_{T}(B)$, so that
\begin{equation}
\label{rhoTB}\rho_{T}\lambda\cdot\lambda^{\prime}=\rho_{T}(B)\lambda
\cdot\lambda^{\prime}=\frac{1}{4}\int_{0}^{T}\langle\Omega_{\lambda}%
B_{s},\Omega_{\lambda^{\prime}}B_{s}\rangle_{H}\,ds.
\end{equation}

\end{notation}

In this notation, (\ref{key-identity-eqn}) reads
\begin{equation}
\mathbb{E}\left[  f(B_{T})e^{i\lambda\cdot Z_{T}}\right]  =\mathbb{E}\left[
f(B_{T})e^{-\frac{1}{2}\rho_{T}\lambda\cdot\lambda}\right]  .
\label{h.equ.4.3}%
\end{equation}

By making the change of variables $s=Ts^{\prime}$ in (\ref{rhoTB}), we get the
scaling relation\footnote{Here we have made use of the fact that $B_{T\left(
\cdot\right)  }\overset{d}{=}\sqrt{T}B_{\left(  \cdot\right)  }.$}
\begin{equation}
\rho_{T}\overset{d}{=}T^{2}\rho_{1}\quad\text{in law}. \label{rhoT-scale}%
\end{equation}

The following essential fact will be proved (in a stronger form) in the next
section; see Corollary \ref{h.the.5.7}.

\begin{proposition}
\label{p.pd}Almost surely, the random linear transformation $\rho_{T}$ is
strictly positive definite.
\end{proposition}

In particular, $\rho_{T}^{-1}$ exists almost surely and is also strictly
positive definite. Given this, we can derive a formula for the heat kernel
$\nu_{T}$.

\begin{theorem}
\label{J0-formula} For any bounded measurable function $F:G\rightarrow
\mathbb{R}$, we have
\begin{equation}
\mathbb{E}[F(g_{T})]=\mathbb{E}\int_{C}F(B_{T},c)J_{T}^{0}(B,c)\,dc
\label{J0-eqn}%
\end{equation}
where $dc$ denotes Lebesgue measure on $C$, and
\[
J_{T}^{0}(B,c):=\frac{\exp\left(  -\frac{1}{2}\rho_{T}(B)^{-1}c\cdot c\right)
}{\sqrt{\det(2\pi\rho_{T}(B))}}.
\]
Moreover, $J_{T}^{0}(B,\cdot)>0$, $\mathbb{P}\otimes m$ -- a.e. where $m$
denotes Lebesgue measure on $C.$
\end{theorem}

\begin{proof}
The fact that $J_{T}^{0}(B,\cdot)>0$, $\mathbb{P}\otimes m$ -- a.e. follows by
Fubini's theorem along with Proposition \ref{p.pd}. If $F$ is of the form
$F(x,c)=f(x)e^{i\lambda\cdot c}$, for some bounded measurable $f:W\rightarrow
\mathbb{R}$ and some $\lambda\in C$, then by (\ref{h.equ.4.3}) and the
Gaussian Fourier transform formula,
\begin{align}
\mathbb{E}\left[  f(B_{T})e^{i\lambda\cdot Z_{T}}\right]   &  =\mathbb{E}%
\left[  f(B_{T})e^{-\frac{1}{2}\rho_{T}\lambda\cdot\lambda}\right] \nonumber\\
&  =\mathbb{E}\left[  f(B_{T})\int_{C}e^{i\lambda\cdot c}\frac{\exp\left(
-\frac{1}{2}\rho_{T}^{-1}c\cdot c\right)  }{\sqrt{\det(2\pi\rho_{T})}}\right]
\nonumber\\
&  =\mathbb{E}\int_{C}f(B_{T})e^{i\lambda\cdot c}J_{T}^{0}(B,c)\,dc
\label{e.jj}%
\end{align}
as desired. Taking $f=1$, $\lambda=0$ in Eq. (\ref{e.jj}) shows that
$\mathbb{E}\int_{C}J_{T}^{0}(B,c)\,dc=1$.

The proof is now easily completed with the help of the multiplicative system
theorem.  Indeed, the
set of all such functions $F(x,c)=f(x)e^{i\lambda\cdot c}$ is a multiplicative
system, and it is standard to show that it generates the Borel $\sigma
$-algebra of $G$. The set of functions $F$ for which (\ref{J0-eqn}) holds is a
vector space. Therefore if $F_{n}$ is a sequence of functions satisfying
(\ref{J0-eqn}) and $F_{n}\rightarrow F$ boundedly, the dominated convergence
theorem shows that $F$ also satisfies (\ref{J0-eqn}). Having verified the
hypotheses of the multiplicative system theorem, we conclude that
(\ref{J0-eqn}) holds for all bounded measurable $F$.
\end{proof}

\begin{corollary}
\label{gamma-formula} The heat kernel measure $\nu_{T}$ is absolutely
continuous to the product measure $d\mu_{T}\otimes dc$, and the Radon--Nikodym
derivative is given by
\[
\gamma_{T}(x,c):=\mathbb{E}\left[  J_{T}^{0}(B,c)\mid B_{T}=x\right]  .
\]
(That is, for each $c$, $\gamma_{T}(\cdot,c)$ is a measurable function on $W$
such that $\gamma_{T}(B_{T},c)=\mathbb{E}\left[  J_{T}^{0}(B,c)\mid
B_{T}\right]  $ almost surely; this function is unique up to $\mu_{T}$-null sets.)
\end{corollary}

A few properties are immediately apparent from this formula.

\begin{corollary}
\label{gamma-positive} We have $\gamma_{T} > 0$, $d\mu_{T} \otimes dc$-almost
everywhere on $G$.
\end{corollary}

\begin{proof}
For each $c \in C$, $J^{0}_{T}(B, c) > 0$ $\mathbb{P}$-almost surely, hence
its conditional expectation $\gamma_{T}(B_{T}, c)$ is also strictly positive
$\mathbb{P}$-almost surely. In other words, for each $c \in C$, we have
$\gamma_{T}(x,c) > 0$ for $\mu_{T}$-almost every $x \in W$. The conclusion
follows by Fubini's theorem.
\end{proof}

\begin{corollary}
\label{inversion-invariant} For any $T>0$, the heat kernel measure $\nu_{T}$
is invariant under the inversion map $g\mapsto g^{-1}$; that is,
\[
\mathbb{E}[F(g_{T})]=\int_{G}F(g)\,d\nu_{T}(g)=\int_{G}F(g^{-1})\,d\nu
_{T}(g)=\mathbb{E}[F(g_{T}^{-1})].
\]
In other words, $g_{T}$ and $g_{T}^{-1}$ have the same law.
\end{corollary}

\begin{proof}
This follows immediately from the observation that $J^{0}_{T}(-B, -c) =
J^{0}_{T}(B,c)$ together with the symmetry of the law of Brownian motion.
\end{proof}

This fact can also be extracted from finite-dimensional approximations; see
\cite[Corollary 4.9]{A35}. It is worth noting that, in contrast to flat
Brownian motion, the \emph{processes} $\{g_{t}\}_{t\ge0}$ and $\{g_{t}%
^{-1}\}_{t\ge0}$ generally do \emph{not} have the same law.

\section{Estimates on $\rho_{T}$}

\label{rhoT-sec}

In this section, we derive technical estimates on the random linear
transformation $\rho_{T}$, which were used in the previous section to define
the heat kernel and will be needed in the sequel for further development of
the smoothness properties of the heat kernel. In particular, we need to show
that $\rho_{T}$ is almost surely invertible, and that its inverse is unlikely
to be large.  Throughout this section, $T > 0$ is fixed, and, for any 
$A:H\rightarrow H$, $\|A\|_{op}$ denotes the standard operator norm of $A$ on
$H$.

\subsection{Small ball estimates}

We will need the following \textquotedblleft small ball\textquotedblright%
\ result, which essentially says that a Brownian motion is unlikely to stay
close to the origin. See \cite[Lemma 2.3]{li-small-ball-Lp} for a proof (of a
more general statement) as well as historical notes.

\begin{theorem}
\label{bt-small-ball} Let $b_{t}$ be a one-dimensional Brownian motion. Then
\[
\lim_{\varepsilon\to0} \varepsilon\log\mathbb{P} \left(  \int_{0}^{1}
b_{t}^{2}\,dt < \varepsilon\right)  = -\frac{1}{8}.
\]
In particular, there is a positive constant $K_{0}$ such that for all
$\varepsilon> 0$,
\[
\mathbb{P} \left(  \int_{0}^{1} b_{t}^{2}\,dt \le\varepsilon\right)  \le K_{0}
e^{-1/4\varepsilon}.
\]

\end{theorem}

\begin{lemma}
\label{ABt-small-ball-exp} Let $A \in HS$. Then for all $\varepsilon> 0$
\[
\mathbb{P} \left(  \int_{0}^{T} \|A B_{t}\|_{H}^{2}\,dt < \varepsilon\right)
\le K_{0} \exp\left(  -\frac{\|A\|_{op}^{2} T^{2}}{4 \varepsilon}
\right)
\]
where $K_{0}$ is the
constant from Theorem \ref{bt-small-ball}.
\end{lemma}

\begin{proof}
By rescaling, it is sufficient to consider the case $T=1$.

By replacing $A$ by $\sqrt{A^{*} A}$, we can assume that $A$ is self-adjoint
and nonnegative definite. In particular, $\lambda:= \|A\|_{op}$ is an
eigenvalue of $A$; let $u$ be a corresponding unit eigenvector. Then $\|A
B_{t}\|_{H}^{2} \ge|\langle AB_{t}, u \rangle|^{2} = \lambda^{2} |\langle
B_{t}, u \rangle|^{2}$. Here we interpret $\langle B_{t}, u_{i} \rangle$ in
the sense of Proposition \ref{ABt-bounded}, viewing $\langle\cdot, u_{i}
\rangle$ as a rank-one Hilbert--Schmidt operator on $H$. Then it is easy to
verify that $\langle B_{t}, u \rangle$ is a one-dimensional standard Brownian
motion. As such,
\begin{align*}
\mathbb{P} \left(  \int_{0}^{1} \|A B_{t}\|_{H}^{2}\,dt < \varepsilon\right)
&  \le\mathbb{P} \left(  \int_{0}^{1} b_{t}^{2}\,dt < \frac{\varepsilon
}{\lambda^{2}}\right) \\
&  \le K_{0} \exp\left(  -\frac{\lambda^{2}}{4 \varepsilon} \right)  .
\end{align*}
As $\lambda= \|A\|_{op}$ the proof is complete.
\end{proof}

We need an analogous statement when the Brownian motion is perturbed by a
one-dimensional drift, under the assumption that $A$ is skew-adjoint. This
will follow from the following lemma from linear algebra.

\begin{lemma}
\label{PA-op} Suppose $A \in HS$ is skew-adjoint, and $P$ is orthogonal
projection onto a subspace of $H$ with codimension 1. Then $\|PA\|_{op} =
\|AP\|_{op} = \|A\|_{op}$.
\end{lemma}

\begin{proof}
The first equality is just the fact that $PA = -(AP)^{*}$. Moreover, the
inequality $\|AP\|_{op} \le\|A\|_{op}$ is obvious.

Since $\sqrt{A^{*} A}$ is compact, self-adjoint and nonnegative definite, it
has $\lambda:= \|\sqrt{A^{*} A}\|_{op} = \|A\|_{op}$ as an
eigenvalue. Let $u$ be a unit eigenvector of $\sqrt{A^{*} A}$ with eigenvalue
$\lambda$. Then $u$ is also an eigenvector of $A^{*} A = -A^{2}$ with
eigenvalue $\lambda^{2}$. Since $Au \ne0$, set $v = Au/\|Au\|$; then we have
$\langle u,v \rangle= 0$ and $-A^{2} v = \lambda^{2} v$.

Let $h_{0}$ be a unit vector in the kernel of $P$, so that $Ph = h - \langle
h, h_{0} \rangle h_{0}$. Now choose a unit vector $w \in\operatorname{span}%
\{u,v\}$ with $\langle w, h_{0} \rangle= 0$. (If $\langle u, h_{0} \rangle=
0$, take $w=u$; else take $w = \langle u, h_{0} \rangle v - \langle v, h_{0}
\rangle u$, appropriately rescaled.) Then $Pw=w$, so $\|APw\|^{2} = \|Aw\|^{2}
= \langle A^{*} A w, w \rangle= \lambda^{2}$. We thus have shown $\|AP\|_{op} 
\ge\lambda= \|A\|_{op}$.
\end{proof}

\begin{lemma}
\label{A-perturb-small-ball} Suppose $A \in HS$ is skew-adjoint. For any fixed
$h_{0} \in H$,
\[
\mathbb{P} \left(  \inf_{\gamma\in C([0,T];H)} \int_{0}^{T} \|A (B_{t} +
\gamma(t) h_{0})\|_{H}^{2}\,dt < \varepsilon\right)  \le K_{0} \exp\left(
-\frac{\|A\|_{op}^{2} T^{2} }{4 \varepsilon}\right),
\]
where $K_{0}$ is the constant from Theorem \ref{bt-small-ball}.
\end{lemma}

\begin{proof}
Let $Ph = h - \|Ah_{0}\|_H^{-1} \langle h, Ah_{0} \rangle h_{0}$ be orthogonal
projection onto the orthogonal complement of $\{A h_{0}\}$. (If $A h_{0} = 0$,
then take $P=I$.) For any $\gamma$, we have
\[
\|A (B_{t} + \gamma(t) h_{0})\|_H^{2} \ge\|PA (B_{t} + \gamma(t) h_{0})\|_H^{2} =
\|PA B_{t}\|_H^{2}.
\]
Thus by Lemma \ref{ABt-small-ball-exp} and Lemma \ref{PA-op}, we have
\begin{align*}
\mathbb{P} \left(  \inf_{\gamma\in C([0,T];H)} \int_{0}^{T} \|A (B_{t} +
\gamma(t) h_{0})\|_{H}^{2}\,dt < \varepsilon\right)   &  \le\mathbb{P} \left(
\int_{0}^{T} \|PA B_{t}\|_{H}^{2}\,dt < \varepsilon\right) \\
&  \le K_{0} \exp\left(  -\frac{\|PA\|_{op}^{2} T^{2} }{4 \varepsilon
}\right) \\
&  = K_{0} \exp\left(  -\frac{\|A\|_{op}^{2} T^{2} }{4 \varepsilon
}\right)  .
\end{align*}

\end{proof}

\subsection{Large deviations for $\|\rho_{T}\|_{op}$}
We will need the following large deviations result for Wiener chaos random
variables, which can be found in \cite[p.6]{ledoux-large-deviations-1990}
together with the relevant definitions, background, and further references.

\begin{theorem}
Let $(\tilde{W}, \tilde{H}, \tilde{\mu})$ be an abstract Wiener space, let $B$
be a real separable Banach space, and let $f \in L^{2}(\tilde{\mu}; B)$ be a
random variable that is in $\mathcal{H}^{(d)}(\tilde{\mu}; B)$, the $B$-valued
homogeneous Wiener chaos of degree $d$. Then
\[
\lim_{r \to\infty} \frac{1}{r^{2/d}} \log\tilde{\mu}\left(  x : \|f(x)\|_{B} >
r\right)  = -\frac{1}{2} \left(  \sup_{h \in\tilde{H}, \|h\|_{\tilde{H}} \le1}
\left\Vert \int_{\tilde{W}} f(x+h)\,\tilde{\mu}(dx) \right\Vert _{B} \right)
^{-2/d}.
\]

\end{theorem}

In particular, as in \cite[p.3]{ledoux-large-deviations-1990}, by the
Cameron--Martin theorem and the Cauchy--Schwarz inequality we have the bound
\begin{equation}
\label{chaos-large-dev}\limsup_{r \to\infty} \frac{1}{r^{2/d}} \log\tilde{\mu
}\left(  x : \|f(x)\|_{B} > r\right)  \le-\frac{e^{-1/d}}{2} \left(
\int_{\tilde{W}} \|f(x)\|_{B}^{2} \,\tilde{\mu}(dx) \right)  ^{-2/d}.
\end{equation}

This bound can be applied to $\rho_{T}$, as the following lemma shows.

\begin{lemma}
\label{rhoT-large-dev} There is a constant $k > 0$, depending only on $\omega
$, such that
\[
\limsup_{r \to\infty} \frac{1}{r} \log\mathbb{P}(\|\rho_{T}\|_{op} > r) \le-\frac
{k}{T^{2}},
\]
where $\|\cdot\|_{op}$ is the operator norm on $\mathrm{End}(C)$ as in Notation \ref{End-notation}.
In particular, for any $k_{1} < k$ there is a $K_{1} > 0$ (depending on
$k_{1}$ and $\omega$) such that
\[
\mathbb{P}(\|\rho_{T}\|_{op} > r) \le K_{1} \exp\left(  -\frac{k_{1} r}{T^{2}}
\right)  .
\]

\end{lemma}

\begin{proof}
By (\ref{rhoT-scale}), we have $\rho_{T}\overset{d}{=}T^{2}\rho_{1}$ in law,
so it suffices to consider $T=1$. We will write $\rho$ for $\rho_{1}$.

Take $\tilde{W} = C([0,1]; W)$ to be the path space over $W$, and $\tilde{\mu
}$ to be the law of Brownian motion $\{B_{t}\}$ on $W$, which is a Gaussian
measure on $\tilde{W}$. Then $(\tilde{W}, \tilde{H}, \tilde{\mu})$ is an
abstract Wiener space, where $\tilde{H}$ is the space of finite-energy
$H$-valued paths in $\tilde{W}$. If we take $E = \operatorname*{End}(C)$ with
the operator norm $\|\cdot\|_{op}$, then we can
consider $\rho$ as an $E$-valued random variable on $\tilde{W}$. It is not hard
to show that $\rho- \mathbb{E} \rho\in\mathcal{H}^{(2)}(\tilde{W}; E)$. So
applying (\ref{chaos-large-dev}) with $d=2$ and adjusting notation, we obtain
\[
\limsup_{r \to\infty} \frac{1}{r} \log\mathbb{P}\left(  \left\Vert \rho-
\mathbb{E} \rho\right\Vert_{op} > r \right)  \le-\frac{e^{-1/2}}{2} \left(
\mathbb{E} \left\Vert \rho- \mathbb{E} \rho\right\Vert_{op} ^{2} \right)  ^{-1}.
\]
We can drop the constant $\mathbb{E} \rho$ from the left side without changing
the limit. Setting
\[
k := \frac{e^{-1/2}}{2} \left(  \mathbb{E} \left\Vert \rho- \mathbb{E}
\rho\right\Vert_{op} ^{2} \right)  ^{-1}%
\]
which only depends on $\omega$, we have the conclusion.
\end{proof}

\subsection{Estimates on $\rho_{T}$}

\begin{notation}
\label{bf-h-notation} Given $h \in H$, let $\mathbf{h}(t) = \frac{t}{T} h$, so
that $\mathbf{h}$ is a finite-energy path in $H \subset W$ with $\mathbf{h}(0)
= 0$ and $\mathbf{h}(T) = h$.
\end{notation}

\begin{lemma}
\label{rhoT-estimate} Fix $T > 0$, $\alpha_{0} \ge0$, and $h \in H$. There are
constants $K, k$, depending on $\omega$, $\alpha_{0}$, $\|h\|_{H}$, and $T$,
so that
\[
\mathbb{P}\left(  \sup_{|\alpha| \le\alpha_{0}} \|\rho_{T}(B + \alpha
\mathbf{h})^{-1}\|_{op} > r\right)  \le K e^{-kr}.
\]

\end{lemma}

\begin{proof}
Let $S$ denote the unit sphere of $C$. For an arbitrary $\delta>0$, we may
cover $S$ with a finite number $n$ of balls of radius at most $\delta$; let
$\{\lambda_{i}\}_{i=1}^{n}$ be their centers. We can choose $n\leq
M\delta^{-d}$, where $d=\dim C$ and $M$ is some universal
constant.\footnote{This crude bound is easily proved when the $\ell^{2}$ norm
on $C$ is replaced by an $\ell^{\infty}$ norm. In this case balls are cubes
and it is just a question of dividing the unit cube into order $\left(
1/\delta\right)  ^{\dim C}$ sub-cubes. Since the $\ell^{2}$ and $\ell^{\infty
}$ norms are equivalent in finite-dimensions it follows that the same
\textquotedblleft entropy\textquotedblright\ estimates hold for round balls.}

For any $Q \in\operatorname*{End}_{+}(C)$ and any $\lambda\in C$, we can
choose a $\lambda_{i}$ with $|\lambda- \lambda_{i}| < \delta$. Then the mean
value theorem gives us $|Q\lambda\cdot\lambda-Q\lambda_{i}\cdot\lambda_{i}|
\leq2\|Q\|_{op}\delta$. Thus
\[
\|Q^{-1}\|_{op}^{-1}=\min_{\lambda\in S}Q\lambda\cdot\lambda\geq\min_{i}Q\lambda
_{i}\cdot\lambda_{i}-2\delta\|Q\|_{op}.
\]
Thus, for any $r > 0$, we have
\begin{align*}
\mathbb{P}  &  \left(  \sup_{|\alpha|\leq\alpha_{0}}\|\rho_{T}(B+\alpha
\mathbf{h})^{-1}\|_{op}>r\right)  =\mathbb{P}\left(  \inf_{|\alpha|\leq\alpha_{0}%
}\inf_{\lambda\in S}\rho_{T}(B+\alpha\mathbf{h})\lambda\cdot\lambda
<r^{-1}\right) \\
&  \leq\mathbb{P}\left(  \inf_{|\alpha|\leq\alpha_{0}}\min_{i}\rho_{T}
(B+\alpha\mathbf{h})\lambda_{i}\cdot\lambda_{i}-2\delta\sup_{|\alpha
|\leq\alpha_{0}}\|\rho_{T}(B+\alpha\mathbf{h})\|_{op}<r^{-1}\right) \\
&  \leq\mathbb{P}\left(  \inf_{|\alpha|\leq\alpha_{0}}\min_{i}\rho_{T}
(B+\alpha\mathbf{h})\lambda_{i}\cdot\lambda_{i}<2r^{-1}\right)  +\mathbb{P}%
\left(  2\delta\sup_{|\alpha|\leq\alpha_{0}}
\|\rho_{T}(B+\alpha\mathbf{h}%
)\|_{op}>r^{-1}\right) \\
&  =:P_{1}+P_{2}.
\end{align*}

To analyze the second term $P_{2}$, let us choose $\delta= 1/r^{2}$, so we
have
\[
P_{2} = P_{2}(r) = \mathbb{P}\left(  2 \sup_{|\alpha|\leq\alpha_{0}}\|\rho
_{T}(B+\alpha\mathbf{h})\|_{op} > r \right)  .
\]
We note that for $|\alpha|\leq\alpha_{0}$, we have
\[
\|\rho_{T}(B+\alpha\mathbf{h})\|_{op}\leq\|\rho_{T}(B)\|_{op}+2\alpha_{0}\|\rho
_{T}(B,\mathbf{h})\|_{op}+\alpha_{0}^{2}\|\rho_{T}(\mathbf{h})\|_{op}.
\]
so for any $\varepsilon> 0$,
\begin{align*}
P_{2}(r)  &  \le\mathbb{P}(2\|\rho_{T}(B)\|_{op} > r/3) + \mathbb{P}(4 \alpha_{0}
\|\rho_{T}(B, \mathbf{h})\|_{op} > r/3) + \mathbb{P} (2 \alpha_{0}^{2} \|\rho
_{T}(\mathbf{h})\|_{op} > r/3)\\
&  =: P_{2,1}(r) + P_{2,2}(r) + P_{2,3}(r).
\end{align*}
Now $P_{2,1}(r)$ is controlled by Lemma \ref{rhoT-large-dev}. For $P_{2,2}$,
we note that $\rho_{T}(B, \mathbf{h})$ is linear in $B$ and hence Gaussian.
Thus by Fernique's theorem \cite{fernique70} (see also \cite[Theorem
3.1]{Kuo75}), we have $P_{2,2}(r) \le K^{\prime}e^{-k^{\prime}r^{2}}$ for some
$K^{\prime}, k^{\prime}$. And since $\rho_{T}(\mathbf{h})$ is deterministic,
$P_{2,3}(r)$ vanishes for all sufficiently large $r$. Thus we have $P_{2}(r)
\le K e^{-kr}$ for suitable $K, k$.

Next we estimate $P_{1}(r) = \mathbb{P}\left(  \inf_{|\alpha|\leq\alpha_{0}%
}\min_{i}\rho_{T}(B+\alpha\mathbf{h})\lambda_{i}\cdot\lambda_{i}%
<2r^{-1}\right)  $. Here $i$ ranges from $1$ to $n = n(r) \le M r^{2d}$, since
we have chosen $\delta= r^{-2}$. Thus by a union bound we have
\begin{align*}
P_{1}(r)  &  \le\sum_{i=1}^{n(r)} \mathbb{P} \left(  \inf_{|\alpha|\leq
\alpha_{0}}\rho_{T}(B+\alpha\mathbf{h})\lambda_{i}\cdot\lambda_{i}%
<2r^{-1}\right) \\
&  \le n(r) \max_{i} \mathbb{P} \left(  \inf_{|\alpha|\leq\alpha_{0}}\rho
_{T}(B+\alpha\mathbf{h})\lambda_{i}\cdot\lambda_{i}<2r^{-1}\right) \\
&  \le M r^{2d} \max_{i} K_{0} \exp\left(  -\frac{r \|\Omega_{\lambda_{i}%
}\|_{op} }{2} \right)
\end{align*}
applying Lemma \ref{A-perturb-small-ball} with $A = \Omega_{\lambda_{i}}$,
$h_{0} = h$, $\gamma(t) = \alpha\frac{t}{T}$, and $\varepsilon= 2r^{-1}$.

Now we note that since $\omega$ is assumed to be surjective, we have
$\Omega_{\lambda}\ne0$ for every $\lambda\ne0$. Then since $S$ is compact and
the map $C \ni\lambda\to\Omega_{\lambda} \in B(H)$ is continuous, we have
$\inf_{\lambda\in S} \|\Omega_{\lambda}\|_{op} > 0$. Thus we have
\[
P_{1}(r) \le K e^{-k r}%
\]
for some (new) constants $K, k$, which have been adjusted so as to absorb the
polynomial factor $M r^{2d}$.

Combining the estimates on $P_{1}(r), P_{2}(r)$ gives the result.
\end{proof}

The previous results and proofs give the following immediate corollary.

\begin{corollary}
\label{h.the.5.7}We have

\begin{enumerate}
\item $\sup_{|\alpha|\leq\alpha_{0}}\|\rho_{T}(B+\alpha\mathbf{h})^{-1}\|_{op}\in
L^{\infty-}(\mathbb{P})$. (Here $L^{\infty-} := \bigcap_{1 \le p < \infty}
L^{p}$.)

\item $(\det\rho_{T})^{-1}=\det(\rho_{T}^{-1})\leq\|\rho_{T}^{-1}\|_{op}^{d}\in
L^{\infty-}$.

\item \label{i.3} Almost surely, $\rho_{T}$ is invertible and hence strictly
positive definite.
\end{enumerate}
\end{corollary}

This was the missing piece in the proof of Theorem \ref{J0-formula} and its corollaries.

\begin{remark}
\label{h.the.5.8}There is a simpler argument to see $\rho_{T}$ is invertible
almost surely when $\omega$ is continuous on $W$. Increments of a $W$-valued
Brownian motion contain, after scaling, an i.i.d. sequence distributed
according to the Gaussian measure $\mu$, which by assumption has full support.
Hence the image of the Brownian motion over times $[0,T]$ is total, almost
surely, so it cannot live in any proper closed subspace of $W$, such as the
kernel of any nonzero continuous operator. We would like to thank George
Lowther \cite{mo-hyperplanes} and Clinton Conley for suggesting this argument.
\end{remark}

\section{A Fernique-type theorem}

\label{fernique-sec}

As an application of the formula in Theorem \ref{J0-formula}, we give a proof
of a Fernique-type theorem (compare \cite{fernique70}) giving
square-exponential integrability for the hypoelliptic heat kernel measure.
This result was previously obtained in \cite[Theorem 4.16]{A35} via
finite-dimensional projections. (Note that \cite[Theorem 4.16]{A35} actually
handles an elliptic heat kernel measure, but the same proof works in the
hypoelliptic case by simply omitting the $B_{0}$ term.)

\begin{proposition}
[Fernique-type theorem]\label{fernique-theorem} There exists $\varepsilon>0$
sufficiently small that, for any $T>0$,
\[
\int_{G} e^{\frac{\varepsilon}{T}(\|x\|_{W}^{2} + |c|_{C})}\,\nu_{T}(dx,dc) =
\mathbb{E}\left[  e^{\frac{\varepsilon}{T}(\Vert B_{T}\Vert_{W}^{2}%
+|Z_{T}|_{C})}\right]  <\infty.
\]

\end{proposition}

\begin{proof}
By the scaling relation (\ref{gt-scale}), we see that it is sufficient to show
the result when $T=1$; then the same $\varepsilon$ will work for every other
$T>0$. As before, we write $\rho$ for $\rho_{1}$.

From Theorem \ref{J0-formula}, we have
\begin{align*}
\mathbb{E}\left[  e^{\varepsilon(\|B_{1}\|_{W}^{2} + |Z_{1}|_{C})}\right]   &
= \mathbb{E} \left[  e^{\varepsilon\|B_{1}\|_{W}^{2}} \int_{C} \frac{
\exp\left(  \varepsilon|c| -\frac{1}{2} \rho^{-1} c \cdot c\right)  }{\sqrt{
\det(2\pi\rho)}}\,dc \right] \\
&  \le\mathbb{E} \left[  e^{\varepsilon\|B_{1}\|_{W}^{2}} \int_{C} \frac
{\exp\left(  \varepsilon\sqrt{\|\rho\|_{op}} |c| -\frac{1}{2} |c|^{2}\right)  }%
{\sqrt{(2\pi)^{d}}}\,dc \right]
\end{align*}
where we made the change of variables $c \to\rho^{1/2} c$ and used the
inequality $|\rho^{1/2} c| \le\sqrt{\|\rho\|_{op}} |c|$. Now letting $a =
\varepsilon\sqrt{\|\rho\|_{op}}$ and writing the $dc$ integral in polar coordinates
gives us the bound
\begin{align*}
\int_{C} \exp\left(  a |c| -\frac{1}{2} |c|^{2}\right)  \,dc  &  =
\omega_{d-1} \int_{0}^{\infty}e^{a r - \frac{1}{2} r^{2}} r^{d-1}\,dr\\
&  \le\omega_{d-1} \left(  \int_{0}^{4 a} e^{ar -\frac{1}{2} r^{2}}
r^{d-1}\,dr + \int_{4a}^{\infty}e^{-\frac{1}{4} r^{2}} r^{d-1}\,dr \right) \\
&  \le\omega_{d-1} \left(  \left(  \sup_{r \ge0} e^{-\frac{1}{2} r^{2}}
r^{d-1} \right)  \int_{0}^{4a} e^{ar}\,dr + \int_{0}^{\infty}e^{-\frac{1}{4}
r^{2}} r^{d-1}\,dr \right) \\
&  \le K e^{4a^{2}}%
\end{align*}
for a suitable constant $K$ not depending on $a$. (Here $\omega_{d-1}$ is the
volume of the $(d-1)$-dimensional unit sphere of $C$.) Thus
\begin{align*}
\mathbb{E}\left[  e^{\varepsilon(\|B_{1}\|_{W}^{2} + |Z_{1}|_{C})}\right]   &
\le K \mathbb{E} \left[  e^{\varepsilon\|B_{1}\|_{W}^{2}} e^{4 \varepsilon^{2}
\|\rho\|_{op}} \right] \\
&  \le K \left(  \mathbb{E} e^{2 \varepsilon\|B_{1}\|_{W}^{2}}\right)  ^{1/2}
\left(  \mathbb{E} e^{8 \varepsilon^{2} \|\rho\|_{op}} \right)  ^{1/2}%
\end{align*}
by the Cauchy--Schwarz inequality. Fernique's theorem asserts that the first
factor is finite for small enough $\varepsilon$. For the second factor, Lemma
\ref{rhoT-estimate} shows that it is finite as soon as $\varepsilon< k$.

\end{proof}

\section{Quasi-invariance of hypoelliptic heat kernel measures}

\label{h.sec.6}

The strict positivity of $\gamma$ combined with the standard Cameron--Martin
theorem implies quasi-invariance for $\nu_{T}$ under translations by
Cameron--Martin subgroup elements.

\begin{proposition}
[Quasi-invariance under right translations I]\label{h.the.6.1} For any $T>0$,
the heat kernel measure $\nu_{T}$ is quasi-invariant under right translation
by elements of the Cameron--Martin subgroup $G_{CM}$. In particular, for any
bounded measurable $F : G \to\mathbb{R}$,
\[
\mathbb{E}[F(g_{T}\cdot g)] = \mathbb{E}[F(g_{T})\mathcal{J}_{g}(g_{T})],
\]
where
\[
\mathcal{J}_{g}(x,c) = \bar{J}_{h}(x) \gamma_{T}\left(  x-h, c-z - \frac{1}{2}
\omega(x,h)\right)  \gamma_{T}(x,c)^{-1}
\]
with
\begin{align*}
\bar{J}_{h}(x) := \exp\left(  \frac{1}{T}\langle h,x \rangle_{H} - \frac
{1}{2T}\|h\|^{2}_{H}\right)  .
\end{align*}

\end{proposition}

\begin{proof}
Let $g =(h,z)\in G_{CM}$. Then by the translation invariance of Lebesgue
measure and the standard Cameron--Martin theorem for $(W,H,\mu_{T})$, we have
that
\begin{align*}
\int_{G} F( (x,  &  c)\cdot(h,z) ) \,d\nu_{T}(x,c)\\
&  = \int_{C} \int_{W} F\left(  x+h, c + z + \frac{1}{2} \omega(x, h)\right)
\gamma_{T}(x,c)\,\mu_{T}(dx)\,dc\\
&  = \int_{C} \int_{W} F\left(  x, c + z + \frac{1}{2} \omega(x-h,h)\right)
\gamma_{T}(x-h, c) \bar{J}_{h}(x) \,\mu_{T}(dx)\,dc\\
&  = \int_{C} \int_{W} F(x, c) \gamma_{T}\left(  x-h, c-z - \frac{1}{2}
\omega(x,h)\right)  \bar{J}_{h}(x) \,\mu_{T}(dx)\,dc.
\end{align*}
Since $\gamma_{T}>0$ $\mu_{T}\otimes dm_{C}$ a.e., this shows that
\begin{multline*}
\int_{G} F( (x,c)\cdot(h,z) ) \,d\nu_{T}(x,c)\\
= \int_{G} F(x, c) \gamma_{T}\left(  x-h, c-z - \frac{1}{2} \omega
(x,h)\right)  \bar{J}_{h}(x) \gamma_{T}(x,c)^{-1}\,d\nu_{T}(x,c).
\end{multline*}

\end{proof}

Alternatively, given the expression of the density $\gamma$ from Corollary
\ref{gamma-formula}, we may reformulate this quasi-invariance result as
follows, which will be more useful in proving subsequent integration by parts formulae.

\begin{proposition}
[Quasi-invariance under right translations II]\label{h.the.6.2} Let
$g=(h,z)\in G_{CM}$. Then, for any measurable $F:G\rightarrow[0,\infty]$ or
measurable $F:G\rightarrow\mathbb{R}$ satisfying $\mathbb{E}|F(g_{T}\cdot
g)|<\infty$, we have that
\[
\mathbb{E}[F(g_{T}\cdot g)] = \int_{C} \,dc\, \mathbb{E}\left[  F(B_{T}%
,c)J_{g}(B,c)\bar{J}_{h}(B_{T})\right]
\]
where, for $\mathbf{h}(t) = \frac{t}{T}h$ as in Notation \ref{bf-h-notation},
\[
J_{g}(B,c) := J^{0}_{T}\left(  B-\mathbf{h},c-z-\frac{1}{2}\omega
(B_{T},h)\right)  ,
\]
and $\bar{J}_{h}$ is as given in Proposition \ref{h.the.6.1}.

\end{proposition}

\begin{proof}
By Theorem \ref{J0-formula}, Fubini's theorem, and the translation invariance
of Lebesgue measure, we have that
\begin{align*}
\mathbb{E} &  [F(g_{T}\cdot g)]=\mathbb{E}\left[  F\left(  B_{T}%
+h,Z_{T}+z+\frac{1}{2}\omega(B_{T},h)\right)  \right]  \\
&  =\int_{C}\,dc\,\mathbb{E}\left[  F\left(  B_{T}+h,c+z+\frac{1}{2}%
\omega(B_{T},h)\right)  \frac{\exp\left(  -\frac{1}{2}\rho_{T}^{-1}c\cdot
c\right)  }{\sqrt{\det(2\pi\rho_{T})}}\right]  \\
&  =\int_{C}\,dc\,\mathbb{E}\left[  F(B_{T}+h,c)\frac{\exp\left(  -\frac{1}%
{2}\rho_{T}^{-1}(c-z-\frac{1}{2}\omega(B_{T},h))\cdot(c-z-\frac{1}{2}%
\omega(B_{T},h))\right)  }{\sqrt{\det(2\pi\rho_{T})}}\right]  .
\end{align*}
Now taking the given finite-energy path $\mathbf{h}$ and translating $B\mapsto
B-\mathbf{h}$, the standard Cameron--Martin theorem on $C([0,T];W)$ for
$\mathrm{Law}(B)$ states that
\begin{align*}
&  \mathbb{E}\left[  F(B_{T}+h,c)\frac{\exp\left(  -\frac{1}{2}\rho_{T}%
^{-1}(c-z-\frac{1}{2}\omega(B_{T},h))\cdot(c-z-\frac{1}{2}\omega
(B_{T},h))\right)  }{\sqrt{\det(2\pi\rho_{T})}}\right]  \\
&  =\mathbb{E}\left[  F(B_{T},c)\frac{\exp\left(  -\frac{1}{2}\rho
_{T}(B-\mathbf{h})^{-1}(c-z-\frac{1}{2}\omega(B_{T},h))\cdot(c-z-\frac{1}%
{2}\omega(B_{T},h))\right)  }{\sqrt{\det(2\pi\rho_{T}(B-\mathbf{h}))}%
}\bar{\mathbf{J}}_{\mathbf{h}}\right]  ,
\end{align*}
where
\[
\bar{\mathbf{J}}_{\mathbf{h}}=\bar{\mathbf{J}}_{\mathbf{h}}(B)=\exp\left(
\int_{0}^{T}\langle\dot{\mathbf{h}}(t),dB_{t}\rangle_{H}-\frac{1}{2}\int%
_{0}^{T}\Vert\dot{\mathbf{h}}(t)\Vert_{H}^{2}\,dt\right)  .
\]
Noting that for the given path $\mathbf{h}$, $\bar{\mathbf{J}}_{\mathbf{h}}$
simplifies to $\bar{J}_{h}$, completes the proof.
\end{proof}

Now, consistent with the notation $J_{g}$ and $\bar{J}_{h}$ defined in
Proposition \ref{h.the.6.2}, we set the following notation for the sequel.

\begin{notation}
\label{h.the.7.4} For any $F=F(B,c)$ and $g=(h,z)\in\mathfrak{g}_{CM}$, we
will write
\[
F_{g}(B,c) := F\left(  B- \mathbf{h},c- z-\frac{1}{2}\omega(B_{T},h)\right)
\]
where $\mathbf{h}(t)=\frac{t}{T}h$ as in Notation \ref{bf-h-notation}. Also,
without further comment, we will make the standard identification between
$\mathfrak{g}_{CM}$ and $G_{CM}$, and for $X\in\mathfrak{g}_{CM}$ we will
write $F_{X}$ to mean the analogous expression to that given above for $F_{g}%
$. Furthermore, we will define
\[
(\tilde{\mathbf{X}}F)(B,c) := \frac{d}{d\varepsilon}\bigg|_{0} F_{\varepsilon
X} (B,c).
\]
Note that when $F=F(B_{T},c)$, then $F_{X}(B_{T},c) = F( (B_{T},c)\cdot-X)$
and $(\tilde{\mathbf{X}}F)(B_{T},c) = -(\tilde{X}F)(B_{T},c)$.
\end{notation}

Although there is a natural group structure on the path space over
$G$, note that the vector field $\tilde{\mathbf{X}}$ is not invariant with respect to
this structure.

With this notation in place, we record the following statement which may be
observed by following the proof of Proposition \ref{h.the.6.2}.

\begin{corollary}
\label{h.the.6.4} Let $g=(h,z)\in G_{CM}$. Then, for any measurable
$F:G\rightarrow\lbrack0,\infty]$ and $\Psi:C([0,T];W)\times C\rightarrow
\lbrack0,\infty]$ we have that
\begin{multline*}
\int_{C} dc \,\mathbb{E}[F( (B_{T},c)\cdot g)\Psi(B,c)J^{0}_{T}(B,c)]\\
=\int_{C}\,dc\,\mathbb{E}\left[  F(B_{T},c)\Psi_{g}(B,c) J_{g}(B,c)\bar{J}%
_{h}(B_{T})\right]  ,
\end{multline*}
for $\mathbf{h}(t)=\frac{t}{T}h$, $J_{g}$, and $\bar{J}_{h}$ as given in
Proposition \ref{h.the.6.2}. Moreover, this equality holds for any measurable
$F:G\rightarrow\mathbb{R}$ and $\Psi:C([0,T];W)\times C\rightarrow\mathbb{R}$
satisfying
\[
\int_{C} dc \,\mathbb{E}[|F( (B_{T},c)\cdot g)\Psi(B,c)|J^{0}_{T}(B,c)]
<\infty.
\]

\end{corollary}

One can directly prove quasi-invariance under left translations in a similar
manner to Propositions \ref{h.the.6.1} or \ref{h.the.6.2}, but it also follows
from quasi-invariance under right translations combined with the invariance of
$\nu_{T}$ under inversions.

\begin{corollary}
[Quasi-invariance under left translations]\label{h.the.6.5} Let $g\in G_{CM}$.
Then, for any measurable $F:G\rightarrow[0,\infty]$ or measurable
$F:G\rightarrow\mathbb{R}$ satisfying $\mathbb{E}|F(g_{T}\cdot g)|<\infty$,
\[
\mathbb{E}[F(g\cdot g_{T})] = \mathbb{E}[F(g_{T}) \tilde{\mathcal{J}}%
_{g}(g_{T})],
\]
where
\[
\tilde{\mathcal{J}}_{g}(g^{\prime}) = \mathcal{J}_{g^{-1}}( (g^{\prime-1})
\]
for all $g^{\prime}\in G$ and $\mathcal{J}_{g}$ as given in Proposition
\ref{h.the.6.1}.
\end{corollary}

\begin{proof}
Let $u(g):=F(g^{-1})$. Then repeatedly applying Corollary
\ref{inversion-invariant} (the invariance of $\nu_{T}$ under inversion) gives
\begin{align*}
\mathbb{E}[F(g\cdot g_{T})]  &  =\mathbb{E}[F(g\cdot g_{T}^{-1})]=\mathbb{E}%
[F(( g_{T}\cdot g^{-1})^{-1})]\\
&  =\mathbb{E}[u( g_{T}\cdot g^{-1})]=\mathbb{E}[u(g_{T})\mathcal{J}_{g^{-1}%
}(g_{T})]\\
&  =\mathbb{E}[u(g_{T}^{-1})\mathcal{J}_{g^{-1}}(g_{T}^{-1})]=\mathbb{E}%
[F(g_{T})\mathcal{J}_{g^{-1}}(g_{T}^{-1})].
\end{align*}

\end{proof}

\section{Smoothness properties of hypoelliptic heat kernel measures}

\label{h.sub.7.1}

In this section, we expand on the quasi-invariance results in the previous
section, and study the integrability and smoothness of the corresponding
Radon--Nikodym derivatives.
Similar techniques were used to handle the elliptic case in \cite{Dobbs2013}.
Again for this section, we fix $T>0$.

First we recall some known results for the standard Cameron--Martin
Radon--Nikodym derivative.

\begin{proposition}
\label{h.the.7.1} For $h\in H$, we have $\bar{J}_{h}(B_{T})\in L^{\infty-}$
and
\[
\sup_{|\varepsilon|\le1} \left|  \frac{d}{d\varepsilon} \bar{J}_{\varepsilon
h}(B_{T})\right|  = \sup_{|\varepsilon|\le1} \left\{  \bar{J}_{\varepsilon
h}(B_{T})\cdot\left|  \frac{1}{T}\langle h,B_{T}\rangle- \frac{\varepsilon}%
{T}\|h\|_{H}^{2}\right|  \right\}  \in L^{\infty-}.
\]

\end{proposition}

These estimates are easily proved, since for example, for all $p\in[1,\infty
)$,
\begin{align*}
\mathbb{E}[\bar{J}_{h}(B_{T})^{p}]  &  = \mathbb{E}\left[  \exp\left(
\frac{p}{T}\langle h,B_{T}\rangle- \frac{p}{2T}\|h\|_{H}^{2}\right)  \right]
= \exp\left(  \frac{1}{2T}(p^{2}-p)\|h\|_{H}^{2} \right)  <\infty.
\end{align*}
Now we need to prove analogous results for $J_{g}$ and its derivatives.

\begin{proposition}
\label{h.the.7.2} For any $p\in\lbrack1,\infty)$ and $g\in G_{CM}$,
\[
\mathbb{E}\left[  \int_{\mathbf{C}}\,dc\,\sup_{|\varepsilon|\leq
1}J_{\varepsilon g}(B,c)^{p}\right]  <\infty.
\]
\end{proposition}

\begin{proof}
For $g=(h,z)$, we have that
\begin{multline*}
\sup_{|\varepsilon|\leq1}J_{\varepsilon g}(B,c) \le\\
\frac{\exp\left(  -\frac{1}{2}\inf_{|\varepsilon|\leq1}\rho_{T}(B-\varepsilon
\mathbf{h})^{-1}\left(  c-\varepsilon z-\frac{1}{2}\varepsilon\omega
(B_{T},h)\right)  \cdot\left(  c-\varepsilon z-\frac{1}{2}\varepsilon
\omega(B_{T},h)\right)  \right)  }{\inf_{|\varepsilon|\leq1}%
\sqrt{\det(2\pi\rho_{T}(B-\varepsilon\mathbf{h}))}}.
\end{multline*}
Lemma \ref{rhoT-estimate} gives the integrability of $\left(  \inf
_{|\varepsilon|\leq1}\det(2\pi\rho_{T}(B-\varepsilon\mathbf{h}))\right)  ^{-p/2}$,
so we need now only deal with the exponential term.

For brevity, set $V = z - \frac{1}{2} \omega(B_{T}, h)$ (as a random element
of $C$); we must show that
\[
\exp\left(  -\frac{p}{2} \inf_{|\varepsilon| \le1} \rho_{T}(B - \varepsilon
\mathbf{h})^{-1} (c - \varepsilon V) \cdot(c - \varepsilon V) \right)  \in
L^{1}(\mathbb{P} \times dc).
\]
Observe the following elementary inequality from linear algebra: if $x \in C$
and $A \in\operatorname*{End}(C)$ is a symmetric positive-definite linear
transformation, then
\begin{align*}
|x|^{2} = |A^{1/2} A^{-1/2} x|^{2} \le \|A^{1/2}\|_{op}^{2} |A^{-1/2} x|^{2} =
\|A\|_{op} (A^{-1} x \cdot x).
\end{align*}
Thus if we set $Y = \sup_{|\varepsilon| \le1} \|\rho_{T}(B - \varepsilon
\mathbf{h})\|_{op}$, we have
\begin{align*}
\rho_{T}(B - \varepsilon\mathbf{h})^{-1} (c - \varepsilon V) \cdot(c -
\varepsilon V) \ge\frac{1}{Y} |c - \varepsilon V|^{2}.
\end{align*}
For $|\varepsilon| \le1$, we have $|c - \varepsilon V| \ge |c| -
\varepsilon |V| \ge |c| - |V|$, and thus
\begin{align*}
|c - \varepsilon V| \ge%
\begin{cases}
0, & |c| < 2 |V| \\
\frac{1}{2}|c|, & |c| \ge2 |V|.
\end{cases}
\end{align*}
Putting all of this together, we have
\begin{align*}
&  \mathbb{E} \left[  \int_{C} \exp\left(  -\frac{p}{2} \inf_{|\varepsilon|
\le1} \rho_{T}(B - \varepsilon\mathbf{h}) (c - \varepsilon V) \cdot(c -
\varepsilon V) \right)  \,dc \right] \\
&  \le\mathbb{E} \left[  \int_{|c| \le2|V|} 1\,dc + \int_{|c| \ge2
|V|} \exp\left(  -\frac{p}{8Y} |c|^{2}\right)  \,dc \right] \\
&  \le\mathbb{E} \left[  m(B(0, 2|V|)) + \left(  \frac{8 \pi Y}{p} \right)
^{d/2} \right]  \le C_{p} \mathbb{E} \left[  |V|^{d} + Y^{d/2}\right]
\end{align*}
where $m$ denotes Lebesgue measure on $C$. To complete the proof, it suffices
to show that $|V|, Y \in L^{\infty-}(\mathbb{P})$. This is straightforward
since
\begin{align*}
\|\rho_{T}(B - \varepsilon\mathbf{h})\| _{op} &  \le\|\rho_{T}(B)\|_{op} + 2 \|\rho_{T}(B,
\mathbf{h})\|_{op} + \|\rho_{T}(\mathbf{h})\|_{op}
\end{align*}
for $|\varepsilon|\le1$, $\rho_{T}(B)$ is in the second homogeneous Wiener
chaos, and $\omega(B_{T}, h)$ and $\rho_{T}(B, \mathbf{h})$ are linear in $B$
and hence Gaussian.
\end{proof}

\begin{definition}
\label{h.the.7.3} A \textbf{polynomial in $A_{1},\ldots,A_{k}\in
\mathrm{End}(C)$ and $c_{1},\ldots,c_{\ell}\in C$} is a function which may be
written as sums of products of factors of the form
\[
A_{i_{1}}\cdots A_{i_{r}} c_{i}\cdot c_{j} \quad\text{ and } \quad
\operatorname{tr}(A_{i_{1}}\cdots A_{i_{r}})
\]
for some $i_{1},\ldots,i_{r}\in\{1,\ldots,k\}$ and $i,j\in\{1,\ldots,\ell\}$.
\end{definition}

\begin{lemma}
\label{h.the.7.5} Given $h_{1},\ldots,h_{m}\in H$ and $X=(h,z)\in
\mathfrak{g}_{CM}$, suppose that $F=F(B,c)$ is polynomial in the matrices
$\rho_{T}(B)^{-1}$ and $\rho_{T}(B,\mathbf{h}_{i})$ and vectors $c$ and
$\omega(B_{T},h_{i})$. Then for any $p\in[1,\infty)$
\[
\mathbb{E}\left[  \int_{C} dc\, \sup_{|\varepsilon|\le1} \left\{  \left|
F_{\varepsilon X}(B,c) \right|  ^{p} J_{\varepsilon X}(B,c) \right\}  \right]
<\infty.
\]

\end{lemma}

\begin{proof}
There exist $K,M<\infty$ such that
\begin{align*}
|F_{\varepsilon X}(B,c)|  &  \le K\bigg(1 + \|\rho_{T}(B-\varepsilon
\mathbf{h})^{-1}\|_{op} + \sum_{i=1}^{m}\bigg(\|\rho_{T}(B-\varepsilon\mathbf{h}%
,\mathbf{h}_{i})\|_{op}\\
&  \qquad\qquad+ |\omega(B_{T}-\varepsilon h,h_{i})|_{C} + \left|
c-\varepsilon z - \frac{1}{2} \omega(B_{T}-\varepsilon h,h_{i}) \right|
_{C}\bigg)\bigg)^{M}.
\end{align*}
Thus,
\begin{multline*}
\sup_{|\varepsilon|\leq1} |F_{\varepsilon X}(B,c)| \le C(h_{i},X)\bigg( 1+
\sup_{|\varepsilon|\leq1} \|\rho_{T}(B-\varepsilon\mathbf{h})^{-1}\|_{op}\\
+ \sum_{i=1}^{m}(\|\rho_{T}(B,\mathbf{h}_{i})\|_{op} + |\omega(B_{T},h_{i})|_C) +
|c|_{C}\bigg) ^{M}%
\end{multline*}
and the result follows from each $\omega(B_{T}, h_{i})$ being Gaussian,
Fernique's Theorem, Lemma \ref{rhoT-estimate}, Proposition \ref{h.the.7.2},
and the fact that
\[
\mathbb{E}\left[  \int_{C} dc\, |c|_{C}^{M} \sup_{|\varepsilon|\leq1}
J_{\varepsilon X}(B,c) \right]  < \infty
\]
by computations analogous to those in the proof of Proposition \ref{h.the.7.2}.
\end{proof}

\begin{proposition}
\label{h.the.7.6} For any $p\in\lbrack1,\infty)$ and $X=(h,z)\in G_{CM}$,
\[
\mathbb{E}\left[  \int_{\mathbf{C}}\,dc\,\sup_{|\varepsilon|\leq1}\left\vert
\frac{d}{d\varepsilon}J_{\varepsilon X}(B,c)\right\vert ^{p}\right]  <\infty.
\]
\end{proposition}

\begin{proof}
First note that for any $\mathbf{x},\mathbf{h}\in C([0,T],H)$,
\begin{align}
\label{e.a1}\frac{d}{d\varepsilon}\rho_{T}(\mathbf{x}-\varepsilon
\mathbf{h})^{-1}  &  =2\rho_{T}(\mathbf{x}-\varepsilon\mathbf{h})^{-1}\rho
_{T}(\mathbf{x}-\varepsilon\mathbf{h},\mathbf{h})\rho_{T}(\mathbf{x}%
-\varepsilon\mathbf{h})^{-1}%
\end{align}
and
\begin{multline}
\label{e.a2}\frac{d}{d\varepsilon} \det(\rho_{T}(\mathbf{x}-\varepsilon
\mathbf{h}))^{-1/2}\\
=\frac{1}{\sqrt{{\det(\rho_{T}(\mathbf{x}-\varepsilon\mathbf{h}))}}}%
\cdot\operatorname{tr}\left( \rho_{T}(\mathbf{x}-\varepsilon\mathbf{h}%
)^{-1}\rho_{T}(\mathbf{x}-\varepsilon\mathbf{h},\mathbf{h})\right) .
\end{multline}
Thus
\[
\frac{d}{d\varepsilon}J_{\varepsilon X}(B,c)=J_{\varepsilon X}(B,c)\cdot
\left\{  \varphi_{\varepsilon}(B,c)+\mathrm{tr}\left(  \rho_{T}(B-\varepsilon
\mathbf{h})^{-1}\rho_{T}(B-\varepsilon\mathbf{h},\mathbf{h})\right)  \right\}
\]
where
\begin{align*}
\varphi_{\varepsilon}(B,c)  &  =\frac{d}{d\varepsilon}\left\{ - \frac{1}%
{2}\rho_{T}(B-\varepsilon\mathbf{h})^{-1}\left(  c-\varepsilon V\right)
\cdot\left(  c-\varepsilon V \right)  \right\} \\
&  =- \rho_{T}(B-\varepsilon\mathbf{h})^{-1}\rho_{T}(B-\varepsilon
\mathbf{h},\mathbf{h})\rho_{T}(B-\varepsilon\mathbf{h})^{-1}\left(
c-\varepsilon V \right)  \cdot\left(  c-\varepsilon V \right) \\
&  \qquad\qquad+\rho_{T}(B-\varepsilon\mathbf{h})^{-1} V \cdot\left(
c-\varepsilon V\right)  .
\end{align*}
where we again have taken $V=z+\frac{1}{2}\omega(B_{T},h)$. We have also used
the symmetry of $\rho_{T}(\mathbf{x})$ to simplify the above expression. Thus,
the proof follows by Proposition \ref{h.the.7.2} and Lemma \ref{h.the.7.5}.
\end{proof}

\begin{remark}
\label{r.7.5} From computations like (\ref{e.a1}) and (\ref{e.a2}), it is
clear that, for $X=(h,z)\in\mathfrak{g}_{CM}$ and $F=F(B,c)$ a polynomial as
in Lemma \ref{h.the.7.5}, the function $\tilde{\mathbf{X}}F$ is again a
polynomial in $\rho_{T}(B)^{-1},\rho_{T}(B,\mathbf{h}_{i}),\rho_{T}%
(B,\mathbf{h})$, $\omega(B_{T},h_{i})$, $\omega(B_{T},h)$, $\omega(h,h_{i})$,
$z$, and $c$. Thus, Lemma \ref{h.the.7.5} is sufficient to show that under the
same assumptions
\begin{multline*}
\mathbb{E}\left[  \int_{C} dc\, \sup_{|\varepsilon|\le1} \left\{  \left|
\frac{d}{d \varepsilon} F_{\varepsilon X}(B,c) \right|  ^{p} J_{\varepsilon
X}(B,c) \right\}  \right] \\
= \mathbb{E}\left[  \int_{C} dc\, \sup_{|\varepsilon|\le1} \left\{  \left|
(\tilde{\mathbf{X}}F)_{\varepsilon X}(B,c) \right|  ^{p} J_{\varepsilon
X}(B,c) \right\}  \right]  <\infty.
\end{multline*}

\end{remark}

\begin{lemma}
\label{h.the.7.8} For any $m\in\mathbb{N}$ and $X_{1},\ldots,X_{m}%
\in\mathfrak{g}_{CM}$, $\tilde{\mathbf{X}}_{m}\cdots\tilde{\mathbf{X}}_{1}\log
J^{0}_{T}(B,c)$ is a polynomial as in Lemma \ref{h.the.7.5}. In particular, we
have that $\tilde{\mathbf{X}}_{m}\cdots\tilde{\mathbf{X}}_{1}\log J^{0}%
_{T}(B,c)\in L^{\infty-}(J_{T}^{0}(B,c)\,d\mathbb{P}\, dc )$.
\end{lemma}

\begin{proof}
For $m=1$, this is essentially Proposition \ref{h.the.7.6}. In particular, the
computations in the proof of Proposition \ref{h.the.7.6} imply that, for
$X=(h,z)\in\mathfrak{g}_{CM}$,
\begin{multline*}
\label{h.equ.7.1}\tilde{\mathbf{X}}\log J^{0}_{T}(B,c) =-\rho_{T}(B)^{-1}%
\rho_{T}(B,\mathbf{h})\rho_{T}(B)^{-1}c\cdot c\\
+\rho_{T}(B)^{-1}c\cdot\left(  z+\frac{1}{2}\omega(B_{T},h)\right)
+\mathrm{tr}\left(  \rho_{T}(B)^{-1}\rho_{T}(B,\mathbf{h})\right)  .
\end{multline*}
In particular, $\tilde{\mathbf{X}}\log J^{0}_{T}(B,c)$ is a polynomial in
$\rho_{T}(B)^{-1},\rho_{T}(B,h)$ and $c,z,\omega(B_{T},h)$. Then the result
follows from Lemma \ref{h.the.7.5} and Remark \ref{r.7.5}.
\end{proof}

\begin{definition}
\label{h.the.7.11} For all $(x,c)\in G$, let $\Vert(x,c)\Vert_{\mathfrak{g}%
}:=\Vert x\Vert_{W}+| c|_{C}$. A function $F:G\rightarrow\mathbb{R}$
is \emph{polynomially bounded} if there exist constants $K,M<\infty$ such
that
\[
|F(g)|\leq K\left(  1+\Vert g\Vert_{\mathfrak{g}}\right)  ^{M}%
\]
for all $g\in G$.
Given $X\in\mathfrak{g}_{CM}$, we say $F$ \emph{is left $X$-differentiable}
if
\[
(\tilde{X}F)(g):=\frac{d}{d\varepsilon}\bigg|_{0}F(g\cdot\varepsilon X)
\]
exists for all $g\in G$. We will say that $F$ is \emph{smooth} if $(\tilde
{X}_{1}\cdots\tilde{X}_{m}F)(g)$ exists for all $m\in\mathbb{N}$,
$X_{1},\ldots,X_{m}\in\mathfrak{g}_{CM}$, and $g\in G$.
\end{definition}

\begin{theorem}
\label{h.the.7.12-1} Let $X=(h,z)\in\mathfrak{g}_{CM}$, $F:G\rightarrow
\mathbb{R}$ be left $X$-differentiable such that $F$ and $\tilde{X}F$ are
polynomially bounded, and $\Psi=\Psi(B,c)$ be a polynomial as in Lemma
\ref{h.the.7.5}. Then
\begin{multline}
\label{e.8.9}\int_{C} dc\, \mathbb{E}[(\tilde{X}F)(B_{T},c)\Psi(B,c)J^{0}%
_{T}(B,c)]\\
= \int_{C} \mathbb{E}[F(B_{T},c)(\tilde{\mathbf{X}}^{*}\Psi)(B,c)J^{0}%
_{T}(B,c)]
\end{multline}
where
\begin{align*}
\tilde{\mathbf{X}}^{*}(B,c) := \tilde{\mathbf{X}}(B,c) + \tilde{\mathbf{X}%
}\log J^{0}_{T}(B,c) + \langle h,B_{T}\rangle
\end{align*}
for $\tilde{\mathbf{X}}$ as in Notation \ref{h.the.7.4}. Furthermore, for any
$X_{1},\ldots,X_{m}\in\mathfrak{g}_{CM}$,
\begin{equation}
\label{e.8.10}\mathbb{E}[(\tilde{X}_{1}\cdots\tilde{X}_{m}F)(g_{T})] \\
= \int_{C} \mathbb{E}[F(B_{T},c)(\tilde{\mathbf{X}}^{*}_{1}\cdots
\tilde{\mathbf{X}}^{*}_{m}1)(B,c)J^{0}_{T}(B,c)].
\end{equation}
where $(\tilde{\mathbf{X}}^{*}_{1}\cdots\tilde{\mathbf{X}}^{*}_{m}1)(B,c)\in
L^{\infty-}(J_{T}^{0}(B,c)\,d\mathbb{P}\,dc)$.
\end{theorem}

\begin{proof}
Since $\tilde{X}F$ is polynomially bounded, there exist $K,M<\infty$ such
that
\begin{multline*}
\sup_{|\varepsilon|\leq1}\left\vert \frac{d}{d\varepsilon}F((B_{T}%
,c)\cdot\varepsilon X)\right\vert =\sup_{|\varepsilon|\leq1}\left\vert
(\tilde{X}F)((B_{T},c)\cdot\varepsilon X)\right\vert \\
\leq\sup_{|\varepsilon|\leq1}K\left(  1+\Vert(B_{T},c)\cdot\varepsilon
X\Vert_{\mathfrak{g}}\right)  ^{M}\leq C(X)\left(  1+\Vert(B_{T}%
,c)\Vert_{\mathfrak{g}}\right)  ^{M},
\end{multline*}
where this last expression is in $L^{\infty-}(J^{0}_{T}(B,c)\,d\mathbb{P}%
\,dc)$ again by Fernique's theorem and arguments similar to those in the proof
of Proposition \ref{h.the.7.2}. Then this implies that
\begin{align*}
\int_{C}\,dc\,\mathbb{E}\Big[ (\tilde{X}F)(B_{T},c) & \Psi(B,c)J^{0}_{T}(B,c)
\Big]\\
&  =\int_{C}\,dc\,\mathbb{E}\left[  \frac{d}{d\varepsilon}\bigg|_{0}%
F((B_{T},c)\cdot\varepsilon X)\Psi(B,c)J^{0}_{T}(B,c) \right] \\
&  =\frac{d}{d\varepsilon}\bigg|_{0}\int_{C}\,dc\,\mathbb{E}\left[
F((B_{T},c)\cdot\varepsilon X)\Psi(B,c)J^{0}_{T}(B,c) \right]  .
\end{align*}
Now applying Corollary \ref{h.the.6.4} gives that
\begin{align*}
\int_{C}\, & dc\,\mathbb{E}\Big[ (\tilde{X}F)(B_{T},c)\Psi(B,c)J^{0}_{T}(B,c)
\Big]\\
&  =\frac{d}{d\varepsilon}\bigg|_{0}\int_{C}\,dc\,\mathbb{E}\left[
F(B_{T},c)\Psi_{\varepsilon X}(B,c)J_{\varepsilon X}(B,c)\bar{J}_{\varepsilon
h}(B_{T})\right] \\
&  =\int_{C}\,dc\,\mathbb{E}\bigg[ F(B_{T},c)\bigg( (\tilde{\mathbf{X}}%
\Psi)(B,c) J^{0}_{T}(B,c)\\
&  \qquad\qquad\qquad\qquad\qquad+ \Psi(B,c) \left(  \tilde{\mathbf{X}}%
J^{0}_{T}(B,c)+J^{0}_{T}(B,c) \frac{d}{d\varepsilon}\bigg|_{0}\bar
{J}_{\varepsilon h}(B_{T})\right)  \bigg) \bigg]\\
&  =\int_{C}\,dc\,\mathbb{E}\left[  F(B_{T},c)\left(  \tilde{\mathbf{X}} +
\tilde{\mathbf{X}}\log J^{0}_{T}(B,c)+\langle h,B_{T}\rangle\right)  \Psi(B,c)
J^{0}_{T}(B,c) \right]
\end{align*}
where this second interchange of differentiation and integration is justified
by Propositions \ref{h.the.7.1} and \ref{h.the.7.6}, Lemma \ref{h.the.7.5},
and Remark \ref{r.7.5}. This completes the proof of (\ref{e.8.9}).

Now, equation (\ref{e.8.9}) implies that that
\begin{align*}
\mathbb{E}[(\tilde{X}_{1}\cdots\tilde{X}_{m}F)(g_{T})]  &  = \int_{C}
\,dc\,\mathbb{E}\left[  (\tilde{X}_{2}\cdots\tilde{X}_{m}F)(B_{T}%
,c)(\tilde{\mathbf{X}}_{1}^{*}1)(B,c) J^{0}_{T}(B,c) \right]
\end{align*}
where $(\tilde{\mathbf{X}}_{1}^{*}1)(B,c)=\tilde{\mathbf{X}}_{1}\log J^{0}%
_{T}(B,c) + \langle h_{1},B_{T}\rangle$. By Lemma \ref{h.the.7.8},
$\tilde{\mathbf{X}}_{1}\log J^{0}_{T}(B,c)$ is a polynomial as in Lemma
\ref{h.the.7.5}. We also have that $\langle h_{1},B_{T}\rangle$ is Gaussian,
and for $X=(h,z)\in\mathfrak{g}_{CM}$, $\tilde{X}\langle h_{1},x\rangle
=\partial_{h}\langle h_{1},x\rangle=\langle h_{1},h\rangle$. Thus, we may
again use (\ref{e.8.9}) and Proposition \ref{h.the.7.1}, along with Lemmas
\ref{h.the.7.5} and \ref{h.the.7.8} and Remark \ref{r.7.5}, and iterative
applications of these gives the desired result. The integrability follows from
Proposition \ref{h.the.7.1} and Lemma \ref{h.the.7.8}.
\end{proof}

The previous theorem and its proof give the following.

\begin{corollary}
[Integration by parts]\label{h.the.7.12-2} Let $X=(h,z)\in\mathfrak{g}_{CM}$
and $F_{1},F_{2}:G\rightarrow\mathbb{R}$ be left $X$-differentiable such that
$F_{1}, F_{2}, \tilde{X}F_{1}, \tilde{X}F_{2}$ are polynomially bounded. Then
\[
\mathbb{E}[(\tilde{X}F_{1})(g_{T})F_{2}(g_{T})]  = \mathbb{E}[F_{1}%
(g_{T})(\tilde{X}^{*}F_{2})(g_{T})]
\]
where
\begin{align*}
\tilde{X}^{*}(x,c) := -\tilde{X}(x,c) + \mathbb{E}[(\tilde{\mathbf{X}}\log
J_{T}^{0})(B,c) J_{T}^{0}(B,c)\,|\,B_{T}=x]\gamma_{T}^{-1}(x,c) +\langle
h,x\rangle.
\end{align*}

\end{corollary}

\begin{proof}
By Theorems \ref{J0-formula} and \ref{h.the.7.12-1}, we have that
\begin{align*}
\mathbb{E}[(\tilde{X}F_{1})(g_{T})F_{2}(g_{T})]  & = \int_{C} dc\, \mathbb{E}[
(\tilde{X}F_{1})(B_{T},c)F_{2}(B_{T},c)J^{0}_{T}(B,c)]\\
& = \int_{C} dc\, \mathbb{E}[ F_{1}(B_{T},c)(\tilde{\mathbf{X}}^{*}%
F_{2})(B_{T},c)J^{0}_{T}(B,c)]\\
& = \int_{C} dc\, \mathbb{E}\Big[F_{1}(B_{T},c)\Big(-(\tilde{X}F_{2}%
)(B_{T},c)\\
& \qquad\quad+ F_{2}(g_{T})\left(  ( \tilde{\mathbf{X}}\log J_{T}^{0})(B,c)+
\langle h,B_{T}\rangle\right) \Big)J^{0}_{T}(B,c)\Big].
\end{align*}

\end{proof}

Theorem \ref{h.the.7.12-1}, in particular, equation (\ref{e.8.10}), also
immediately gives the following result for higher order derivatives.

\begin{corollary}
\label{h.the.7.12} Let $m\in\mathbb{N}$ and $X_{1},\ldots,X_{m}\in
\mathfrak{g}_{CM}$. Then, for any smooth $F:G\rightarrow\mathbb{R}$ such that
$F$ and all of its derivatives are polynomially bounded,
\begin{align}
\label{e.8.12}\mathbb{E}[(\tilde{X}_{1}\cdots\tilde{X}_{m}F)(g_{T})] =
\mathbb{E}[F(g_{T})\psi^{X_{m},\ldots,X_{1}}(g_{T})],
\end{align}
where
\begin{align*}
\psi^{X_{m},\ldots,X_{1}}(x,c)  & = \mathbb{E}\left[  (\tilde{\mathbf{X}}%
_{m}^{*}\cdots\tilde{\mathbf{X}}_{1}^{*}1)(B,c) J^{0}_{T}(B, c)\,\big|\,
B_{T}=x\right]  \gamma_{T}^{-1}(x,c)
\end{align*}
and $\psi^{X_{m},\ldots,X_{1}}\in L^{\infty-}(\nu_{T})$.
\end{corollary}

\begin{proof}
Equation (\ref{e.8.12}) and the expression given for $\psi^{X_{m},\ldots
,X_{1}}$ are a direct consequence of Theorem \ref{h.the.7.12-1}. Now note that
the integrability of $\tilde{\mathbf{X}}_{m}^{*}\cdots\tilde{\mathbf{X}}%
_{1}^{*}1$ implies that $\psi^{X_{m},\ldots,X_{1}}\in L^{1}(\nu_{T})$ and
$F\psi^{X_{m},\ldots,X_{1}}\in L^{1}(\nu_{T})$ for any such $F$. More
particularly, for any fixed $p\in(1,\infty)$, we have that
\begin{align*}
\mathbb{E} & |F(g_{T})\psi^{X_{m},\ldots,X_{1}}(g_{T})|\\
& = \int_{C} dc\, \mathbb{E}\left[ \left| F(B_{T},c) \mathbb{E}\left[
(\tilde{\mathbf{X}}_{m}^{*}\cdots\tilde{\mathbf{X}}_{1}^{*}1)(B,c) J^{0}%
_{T}(B, c)\,\big|\, B_{T}\right] \right|  \gamma_{T}^{-1}(B_{T},c)J_{T}%
^{0}(B,c)\right] \\
& \le\int_{C} dc\, \mathbb{E}\left[ \left| F(B_{T},c) (\tilde{\mathbf{X}}%
_{m}^{*}\cdots\tilde{\mathbf{X}}_{1}^{*}1)(B,c)\right|  J^{0}_{T}(B, c)\right]
\\
& \le\|F\|_{L^{p}(\nu_{T})} \|\tilde{\mathbf{X}}_{m}^{*}\cdots\tilde
{\mathbf{X}}_{1}^{*}1\|_{L^{q}(J_{T}^{0}(B,c)\,d\mathbb{P}\,dc)}%
\end{align*}
where $q$ is the conjugate exponent to $p$. As this bound holds, for example,
for all smooth cylinder functions $F$ and these comprise a dense subspace of
$L^{p}(\nu_{T})$, this implies that $\psi^{X_{m},\ldots,X_{1}}$ represents a
bounded linear operator on $L^{p}(\nu_{T})$ and thus $\psi^{X_{m},\ldots
,X_{1}}\in(L^{p}(\nu_{T}))^{*}\cong L^{q}(\nu_{T})$. As this holds for any
$p\in(1,\infty)$, we have that $\psi^{X_{m},\ldots,X_{1}}\in\cap
_{q\in[1,\infty)} L^{q}(\nu_{T})$.
\end{proof}

Right integration by parts formulae may be proved directly, but also follow
from the left integration by parts formulae combined with the invariance of
the heat kernel measure under inversions. For the following corollary, let
$\hat{X}$ denote the right-invariant vector field given by
\[
\hat{X}F(g) = \partial_{(h,z)}(F\circ R_{g})(\mathbf{e})
\]
where $R_{g}$ is right translation by $g\in G$. A straightforward computation
shows that
\[
\hat{X}F(x,c) = \partial_{(h,z-\frac{1}{2}\omega(x,h))}F(x,c);
\]
(compare with Notation \ref{vector-fields} for left-invariant vector fields).

\begin{corollary}
\label{h.the.7.14} Under the hypotheses of Corollary \ref{h.the.7.12},
\[
\mathbb{E}[(\hat{X}_{1}\cdots\hat{X}_{m}F)(g_{T})] =\mathbb{E}[F(g_{T}%
)\hat{\psi}^{X_{m},\ldots,X_{1}}(g_{T})],
\]
where
\[
\hat{\psi}^{X_{m},\ldots,X_{1}}(g) := (-1)^{m}\psi^{X_{m},\ldots,X_{1}}%
(g^{-1}).
\]

\end{corollary}

\begin{proof}
Take $u(g):=F(g^{-1})$. We proceed by induction. Note first that, for any
$g\in G$ and $X\in\mathfrak{g}_{CM}$,
\begin{equation}
\label{h.equ.7.3}(\hat{X}F)(g) = \frac{d}{d\varepsilon}\bigg|_{0}
F(\varepsilon X\cdot g) = \frac{d}{d\varepsilon}\bigg|_{0} u(g^{-1}%
\cdot-\varepsilon X) = -(\tilde{X}u)(g^{-1}).
\end{equation}
This equation and repeated applications of Corollary \ref{inversion-invariant}
give that
\begin{align*}
\mathbb{E}[(\hat{X}F)(g_{T})]  &  = -\mathbb{E}[(\tilde{X} u)(g_{T}^{-1})] =
-\mathbb{E}[(\tilde{X} u)(g_{T})]\\
&  = -\mathbb{E}[u(g_{T})\psi^{X}(g_{T})] = -\mathbb{E}[F(g_{T}^{-1})\psi
^{X}(g_{T})]\\
&  = -\mathbb{E}[F(g_{T})\psi^{X}(g_{T}^{-1})],
\end{align*}
where we have applied Theorem \ref{h.the.7.12} in the third equality. Now
assuming the formula for $m$ and again using equation (\ref{h.equ.7.3}),
Corollary \ref{inversion-invariant}, and Theorem \ref{h.the.7.12} gives
\begin{align*}
\mathbb{E} \left[  (\hat{X}_{1}\cdots\hat{X}_{m+1}F)(g_{T})\right]   &  =
(-1)^{m} \mathbb{E}\left[  (\hat{X}_{m+1}F)(g_{T}) \psi^{X_{m},\ldots,X_{1}%
}(g_{T}^{-1})\right] \\
&  = (-1)^{m+1}\mathbb{E}\left[  (\tilde{X}_{m+1}u)(g_{T}^{-1}) \psi
^{X_{m},\ldots,X_{1}}(g_{T}^{-1})\right] \\
&  = (-1)^{m+1}\mathbb{E}\left[  (\tilde{X}_{m+1}u)(g_{T}) \psi^{X_{m}%
,\ldots,X_{1}}(g_{T})\right] \\
&  = (-1)^{m+1}\mathbb{E}\left[  u(g_{T}) \psi^{X_{m+1},\ldots,X_{1}}
(g_{T})\right] \\
&  = (-1)^{m+1}\mathbb{E}\left[  F(g_{T}^{-1}) \psi^{X_{m+1},\ldots,X_{1}%
}(g_{T})\right] \\
&  = (-1)^{m+1}\mathbb{E}\left[  F(g_{T}) \psi^{X_{m+1},\ldots,X_{1}}%
(g_{T}^{-1})\right]  .
\end{align*}

\end{proof}

\section{Conclusion}

We have shown that the hypoelliptic heat kernel measure $\nu_{T}$ on $G$ is
smooth, in a sense that naturally extends the well-known smoothness results in
finite-dimensions; namely, it is quasi-invariant under left and right
translations by elements of the Cameron--Martin subgroup $G_{CM}$.

In flat abstract Wiener space $(W,H,\mu)$, the smoothness of Gaussian measure
under translation by $H$ (established by the Cameron--Martin theorem) is the
starting point for defining the gradient operator on $L^{2}(W,\mu)$, the
associated Sobolev spaces, chaos decompositions, the Skorohod integral, and
many other developments. Similar results should be possible in our
hypoelliptic setting, and we hope in the future to explore some of this territory.

In this paper, we have considered only groups $G$ whose center $C$ is
finite-dimensional, and our argument makes essential use of this assumption in
several places. It would be interesting to relax this assumption, to allow for
infinite-dimensional centers. For example, the definition of $G$ makes sense
if $C$ is replaced by a separable Hilbert space. However, the lack of a
natural reference measure on $C$ seems to be a significant obstruction to
proving analogous results in this case.

In another direction, it would be interesting to consider a more general class
of groups; for example, nilpotent Lie groups of step 3 or higher.
Unfortunately, in such examples, our approach in Section \ref{h.sec.2} no
longer succeeds; the commutation of terms analogous to $S$ and $L$ may fail.
It appears that new ideas may be required to proceed beyond step 2.

{\bf Acknowledgments.} The authors would like to thank Clinton Conley and Leonard Gross for
helpful discussions, and Fabrice Baudoin for pointing us to the works
of M.~Yor and others cited in Section \ref{h.sec.2}.

\appendix

\section{Another proof of Theorem \ref{h.the.2.1}}
\label{appendix}

In this section, we provide another self-contained proof of Theorem \ref{h.the.2.1},
which is based on the analysis of the infinitesimal generator of
$g_t$.  We will begin with a slightly informal version of the proof and then fill in
the missing technical points.

Recall that $\left\{  B_{t}\right\}  _{t\geq0}$ is an $N$-dimensional
Brownian motion, $A$ is an $N\times N$ skew-symmetric matrix, and we
wish to show for any measurable $f:\mathbb{R}^{N}\rightarrow\mathbb{C}$ such that
$\mathbb{E}\left\vert f\left(  B_{T}\right)  \right\vert <\infty$, we have
\begin{equation}
\mathbb{E}\left[  f\left(  B_{T}\right)  e^{i\int_{0}^{T}AB_{t}\cdot dB_{t}%
}\right]  =\mathbb{E}\left[  f\left(  B_{T}\right)  e^{-\frac{1}{2}\int%
_{0}^{T}\left\vert AB_{t}\right\vert ^{2}dt}\right]  . \label{h.equ.2.1.app}%
\end{equation}

Suppose that $F:\mathbb{R}\times\mathbb{R}^{N}\rightarrow\mathbb{R}$ is a
$C^{2}$-function such that $F$ and its derivatives up to order $2$ have at
most polynomial growth, and set
\[
Y_{t}:=e^{i\int_{0}^{t}AB_{\tau}\cdot dB_{\tau}}.
\]
Then by It\^{o}'s formula
\[
dY_{t}=Y_{t}\left(  iAB_{t}\cdot dB_{t}-\frac{1}{2}\left\vert AB_{t}%
\right\vert ^{2}dt\right)  ,
\]
and%
\begin{multline*}
d\left[  F\left(  t,B_{t}\right)  Y_{t}\right]  =Y_{t}\left(  \nabla
F_{t}\left(  t,B_{t}\right)  +iF\left(  t,B_{t}\right)  AB_{t}\right)  \cdot
dB_{t}\\
+Y_{t}\left(  F_{t}\left(  t,B_{t}\right)  +\frac{1}{2}\Delta F\left(
t,B_{t}\right)  -\frac{1}{2}\left\vert AB_{t}\right\vert ^{2}F\left(
t,B_{t}\right)  +iAB_{t}\cdot\nabla F_{t}\left(  t,B_{t}\right)  \right)  dt.
\end{multline*}
(Here $\nabla F(t,x)$ denotes the gradient with respect to the $x$ variable
only.) By our assumptions on $F$ it is easily verified that
\[
\mathbb{E}\left[  \int_{0}^{T}\left\vert Y_{t}(\nabla F_{t}\left(
t,B_{t}\right)  +iF\left(  t,B_{t}\right)  AB_{t})\right\vert ^{2}dt\right]
<\infty,
\]
and hence $\int_{0}^{\cdot}Y_{t}\left(  \nabla F_{t}\left(  t,B_{t}\right)
+iF\left(  t,B_{t}\right)  AB_{t}\right)  \cdot dB_{t}$ is a square integrable
martingale. Therefore it follows that
\begin{equation}%
\begin{split}
&  \frac{d}{dt}\mathbb{E}\left[  F\left(  t,B_{t}\right)  Y_{t}\right] \\
&  =\mathbb{E}\left[  Y_{t}\left(  F_{t}\left(  t,B_{t}\right)  +\frac{1}%
{2}\Delta F\left(  t,B_{t}\right)  -\frac{1}{2}\left\vert AB_{t}\right\vert
^{2}F\left(  t,B_{t}\right)  +iAB_{t}\cdot\nabla F\left(  t,B_{t}\right)
\right)  \right]  .
\end{split}
\label{e.dt}%
\end{equation}

Now suppose that $f:\mathbb{R}^{N}\rightarrow\mathbb{R}$ and $T>0$ are given
such that there exists a function $F$ as above, with the additional properties
that $F\left(  T,x\right)  =f\left(  x\right)  $ and
\[
F_{t}\left(  t,x\right)  +\frac{1}{2}\Delta F\left(  t,x\right)  -\frac{1}%
{2}\left\vert Ax\right\vert ^{2}F\left(  t,x\right)  +iAx\cdot\nabla F\left(
t,x\right)  =0
\]
for all $\left(  t,x\right)  $. It then follows from (\ref{e.dt}) that
$\frac{d}{dt}\mathbb{E}\left[  F\left(  t,B_{t}\right)  Y_{t}\right]  =0$ and,
in particular,%
\begin{equation}
\mathbb{E}\left[  f\left(  B_{T}\right)  Y_{T}\right]  =\mathbb{E}\left[
F\left(  T,B_{T}\right)  Y_{T}\right]  =\mathbb{E}\left[  F\left(
0,B_{0}\right)  Y_{0}\right]  =F\left(  0,0\right)  . \label{h.equ.2.3}%
\end{equation}
Formally, the solution to (\ref{e.dt}) is given by%
\[
F\left(  t,x\right)  =\left(  e^{\left(  T-t\right)  \left(  L+S\right)
}f\right)  \left(  x\right)
\]
where
\begin{align*}
\left(  Lf\right)  \left(  x\right)   &  :=\frac{1}{2}\Delta f\left(
x\right)  -\frac{1}{2}\left\vert Ax\right\vert ^{2}f\left(  x\right)  \text{
and }\\
\left(  Sf\right)  \left(  x\right)   &  :=iAx\cdot\nabla f\left(  x\right)  .
\end{align*}
With this notation we may summarize (\ref{h.equ.2.3}) as%
\[
\mathbb{E}\left[  f\left(  B_{T}\right)  Y_{T}\right]  =\left(  e^{T\left(
L+S\right)  }f\right)  \left(  0\right)  .
\]

Since $e^{-itS}f\left(  x\right)  =f\left(  e^{tA}x\right)  $ where $e^{tA}$
is a rotation and $\Delta$ is invariant under rotations, it follows that
\[
\left(  e^{-itS}\Delta f\right)  \left(  x\right)  =\left(  \Delta f\right)
\left(  e^{tA}x\right)  =\Delta\left(  f\circ e^{tA}\right)  \left(  x\right)
=\left(  \Delta e^{-itS}f\right)  \left(  x\right)  .
\]
Differentiating this equation in $t$ then shows $\left[  S,\Delta\right]  =0.$
Also,
\begin{align*}
\left(  \left\vert Ay\right\vert ^{2}f\left(  y\right)  \right)
\bigg|_{y=e^{tA}x}  &  =\left\vert Ae^{tA}x\right\vert ^{2}f\left(
e^{tA}x\right)  =\left\vert e^{tA}Ax\right\vert ^{2}f\left(  e^{tA}x\right) \\
&  =\left\vert Ax\right\vert ^{2}f\left(  e^{tA}x\right)  .
\end{align*}
Equivalently, for $M_{g}$ multiplication by $g$, $e^{-itS}M_{\left\vert
A\left(  \cdot\right)  \right\vert ^{2}}=M_{\left\vert A\left(  \cdot\right)
\right\vert ^{2}}e^{-itS}$, and differentiating this at $t=0$ implies%
\[
\left[  S,M_{\left\vert A\left(  \cdot\right)  \right\vert ^{2}}\right]  =0.
\]
Combining this with $\left[  S,\Delta\right]  =0$ allows us to conclude
$\left[  S,L\right]  =0$ and therefore, formally,
\[
\mathbb{E}\left[  f\left(  B_{T}\right)  Y_{T}\right]  =\left(  e^{T\left(
L+S\right)  }f\right)  \left(  0\right)  =\left(  e^{TS}e^{TL}f\right)
\left(  0\right)  =\left(  e^{TL}f\right)  \left(  0\right)
\]
wherein the last equality we have used the fact that $Sf\left(  0\right)  =0$
for all functions $f.$ Hence, by an application of the Feynman--Kac formula, we
conclude that (\ref{h.equ.2.1.app}) holds.

In order to make the above argument rigorous, it is helpful to find
finite-dimensional subspaces of functions which are invariant under the
actions of $L$ and $S.$ So suppose for the moment that $A$ is non-degenerate,
in which case $L$ is a harmonic oscillator Hamiltonian. If $\psi\in
L^{2}\left(  \mathbb{R}^{n}\right)  $ is an eigenvector of $L$ with eigenvalue
$\lambda,$ then
\[
LS\psi=SL\psi=S\lambda\psi=\lambda S\psi.
\]
Given any $\Lambda\in\left(  0,\infty\right)  $, we let $K_{\Lambda}$ be the
linear combination of eigenfunctions of $L$ with eigenvalues $\lambda
\leq\Lambda.$ Then $K_{\Lambda}$ is a finite-dimensional subspace of
$\mathcal{S}\left(  \mathbb{R}^{N}\right)  \subset L^{2}\left(  \mathbb{R}%
^{N}\right)  $ which is invariant under the actions of $L$ and $S.$ For $f\in
K_{\Lambda}$ all of the manipulations in the previous paragraph are justified
and therefore (\ref{h.equ.2.1.app}) holds for all $f\in\cup_{\Lambda<\infty
}K_{\Lambda}$, which is dense in $\mathcal{S}(\mathbb{R}^{N}).$ The full
result for non-degenerate $A$ then follows by density arguments.

We now have a couple of choices for how to proceed when $A\ $is degenerate.
The first is to decompose $\mathbb{R}^{N}$ in $\operatorname*{Nul}\left(
A\right)  \oplus\operatorname*{Nul}\left(  A\right)  ^{\perp}$ and then
decompose the Brownian motion accordingly. With this decomposition it
basically then suffices to prove (\ref{h.equ.2.1.app}) in two cases corresponding
to $A=0$ and to $A$ being non-degenerate. As the case where $A=0$ is a
triviality, the argument is essentially complete. An alternative is to modify
the method of the previous paragraph so as to work for general skew-symmetric
$A$, and this is what we do now.

We start by finding the \textquotedblleft ground state\textquotedblright\ for
$L$ in the form $\Phi\left(  x\right)  =\exp\left(  -\frac{1}{2}\Sigma x\cdot
x\right)  $ for some symmetric non-negative $N\times N$ matrix $\Sigma.$ A
simple computation shows%
\[
2L\Phi=\Phi\cdot\left(  \nabla\cdot\left(  -\Sigma x\right)  +\left\vert
\Sigma x\right\vert ^{2}-\left\vert Ax\right\vert ^{2}\right)  =\Phi
\cdot\left(  -\operatorname*{tr}\left(  \Sigma\right)  +\left\vert \Sigma
x\right\vert ^{2}-\left\vert Ax\right\vert ^{2}\right)  .
\]
Thus taking $\Sigma:=\sqrt{A^{\ast}A}=\sqrt{-A^{2}}$ then implies
$L\Phi=-\frac{1}{2}\operatorname*{tr}\left(  \Sigma\right)  \Phi.$ Further
observe that
\[
S\Phi=iAx\cdot\nabla\Phi=-i\left(  Ax\cdot\Sigma x\right)  \Phi=0,
\]
wherein we have used that $\left[  A,\Sigma\right]  =0$ implies that
\[
Ax\cdot\Sigma x=-x\cdot A\Sigma x=-x\cdot\Sigma Ax=-\Sigma x\cdot Ax,
\]
and hence $Ax\cdot\Sigma x=0.$ Now let $\mathcal{P}_{m}$ denote the space of
polynomial functions on $\mathbb{R}^{N}$ with degrees less than or equal to
$m$, and let $\Phi\mathcal{P}_{m}:=\{\Phi p:p\in\mathcal{P}_{m}\}$. It then
follows that\footnote{The space $\Phi\mathcal{P}_{m}$ is a spectral subspace
as described above in the case that $A$ is non-degenerate.} $\Phi
\mathcal{P}_{m}$ is a finite-dimensional subspace of functions on
$\mathbb{R}^{N}$ which are invariant under the action of both $L$ and $S.$
Indeed, for any $g\in C^{\infty}\left(  \mathbb{R}^{N}\right)  $ we have%
\begin{align*}
L\left(  \Phi g\right)   &  = (L\Phi)g+ \Phi\cdot\frac{1}{2}\Delta
g+\nabla\Phi\cdot\nabla g\\
&  = \Phi\cdot\left(  -\frac{1}{2}\operatorname*{tr}\left(  \Sigma\right)
g+\frac{1}{2}\Delta g+\nabla\ln\Phi\cdot\nabla g\right) \\
&  =\Phi\cdot\left(  \frac{1}{2}\Delta-\Sigma x\cdot\nabla-\frac{1}%
{2}\operatorname*{tr}\left(  \Sigma\right)  \right)  g
\end{align*}
and
\[
S\left(  \Phi g \right)  =\Phi\cdot Sg.
\]
Thus, for any $f\in\Phi\mathcal{P}_{m}$, we have
\[
\mathbb{E}\left[  f\left(  B_{T}\right)  Y_{T}\right]  =\left(  e^{TL}%
f\right)  \left(  0\right)  =\mathbb{E}\left[  f\left(  B_{T}\right)
e^{-\frac{1}{2}\int_{0}^{T}\left\vert AB_{t}\right\vert ^{2}dt}\right]  ,
\]
where the second equality follows from the fact that
\[
d\left[  e^{-\frac{1}{2}\int_{0}^{t}\left\vert AB_{\tau}\right\vert ^{2}d\tau
}\left(  e^{\left(  T-t\right)  L}f\right)  \left(  B_{t}\right)  \right]
=e^{-\frac{1}{2}\int_{0}^{t}\left\vert AB_{\tau}\right\vert ^{2}d\tau}\left(
\nabla e^{\left(  T-t\right)  L}f\right)  \left(  B_{t}\right)  \cdot dB_{t}%
\]
by It\^{o}'s lemma, and thus
\[
\mathbb{E}\left[  e^{-\frac{1}{2}\int_{0}^{t}\left\vert AB_{\tau}\right\vert
^{2}d\tau}\left(  e^{\left(  T-t\right)  L}f\right)  \left(  B_{t}\right)
\right]
\]
is constant in $t$; comparing the values of this expression at $t=T$ and $t=0$
then gives the desired equality.

For any $z\in\mathbb{C}^{N},$
\begin{align*}
\mathbb{E}\left\vert e^{z\cdot B_{T}}-\sum_{n=0}^{K}\frac{\left(  z\cdot
B_{T}\right)  ^{n}}{n!}\right\vert  &  =\mathbb{E}\left\vert \sum
_{n=K+1}^{\infty}\frac{\left(  z\cdot B_{T}\right)  ^{n}}{n!}\right\vert \\
&  \leq\mathbb{E}\left[  \sum_{n=K+1}^{\infty}\frac{\left\vert z\right\vert
^{n}\cdot\left\vert B_{T}\right\vert ^{n}}{n!}\right]  \rightarrow0\text{ as
}K\rightarrow\infty
\end{align*}
by the dominated convergence theorem because the last integrand goes to $0$ as
$K\rightarrow\infty$ and satisfies the estimate%
\[
\sum_{n=K+1}^{\infty}\frac{\left\vert z\right\vert ^{n}\cdot\left\vert
B_{T}\right\vert ^{n}}{n!}\leq e^{\left\vert z\right\vert \left\vert
B_{T}\right\vert }\in L^{1}\left(  P\right)  .
\]
Using the previous observation we may now conclude
\begin{align*}
\mathbb{E}\left[  e^{z\cdot B_{T}}\Phi\left(  B_{T}\right)  Y_{T}\right]   &
=\lim_{K\rightarrow\infty}\mathbb{E}\left[  \sum_{n=0}^{K}\frac{\left(  z\cdot
B_{T}\right)  ^{n}}{n!}\Phi\left(  B_{T}\right)  Y_{T}\right] \\
&  =\lim_{K\rightarrow\infty}\mathbb{E}\left[  \sum_{n=0}^{K}\frac{\left(
z\cdot B_{T}\right)  ^{n}}{n!}\Phi\left(  B_{T}\right)  e^{-\frac{1}{2}%
\int_{0}^{T}\left\vert AB_{t}\right\vert ^{2}dt}\right] \\
&  =\mathbb{E}\left[  e^{z\cdot B_{T}}\Phi\left(  B_{T}\right)  e^{-\frac
{1}{2}\int_{0}^{T}\left\vert AB_{t}\right\vert ^{2}dt}\right]
\end{align*}
holds for all $z\in\mathbb{C}^{N}$. Restricting $z$ to be in $i\mathbb{R}%
^{N},$ we may apply Dynkin's multiplicative system theorem \cite[Appendix A,
p. 309]{Janson1997} in order to show
\[
\mathbb{E}\left[  u\left(  B_{T}\right)  \Phi\left(  B_{T}\right)
Y_{T}\right]  =\mathbb{E}\left[  u\left(  B_{T}\right)  \Phi\left(
B_{T}\right)  e^{-\frac{1}{2}\int_{0}^{T}\left\vert AB_{t}\right\vert ^{2}%
dt}\right]
\]
for all bounded measurable functions $u$ on $\mathbb{R}^{N}$. Given
$f:\mathbb{R}^{N}\rightarrow\mathbb{C}$ such that $\mathbb{E}\left\vert
f\left(  B_{T}\right)  \right\vert <\infty$ and $m\in\mathbb{N}$, apply the
previous formula with $u(x)=u_{m}\left(  x\right)  =\Phi^{-1}f\cdot
1_{\left\vert \Phi^{-1}f\right\vert \leq m}$ to learn%
\[
\mathbb{E}\left[  f\left(  B_{T}\right)  \cdot1_{\left\vert \Phi
^{-1}f\right\vert \left(  B_{T}\right)  \leq m}Y_{T}\right]  =\mathbb{E}%
\left[  f\left(  B_{T}\right)  \cdot1_{\left\vert \Phi^{-1}f\right\vert
\left(  B_{T}\right)  \leq m}e^{-\frac{1}{2}\int_{0}^{T}\left\vert
AB_{t}\right\vert ^{2}dt}\right]  .
\]
Now use the dominated convergence theorem to let $m\rightarrow\infty$ and
finish the proof.

\def\cprime{$'$}
\providecommand{\bysame}{\leavevmode\hbox to3em{\hrulefill}\thinspace}
\providecommand{\MR}{\relax\ifhmode\unskip\space\fi MR }
\providecommand{\MRhref}[2]{%
  \href{http://www.ams.org/mathscinet-getitem?mr=#1}{#2}
}
\providecommand{\href}[2]{#2}


\begin{thebibliography}{10}

\bibitem{abgr-intrinsic}
Andrei Agrachev, Ugo Boscain, Jean-Paul Gauthier, and Francesco Rossi,
  \emph{The intrinsic hypoelliptic {L}aplacian and its heat kernel on
  unimodular {L}ie groups}, J. Funct. Anal. \textbf{256} (2009), no.~8,
  2621--2655.

\bibitem{mse-hilbert-schmidt}
Mart\'{i}n Argerami, \emph{Approximating a {H}ilbert-{S}chmidt operator},
  Mathematics Stack Exchange, URL:http://math.stackexchange.com/q/373117
  (version: 2013-04-26).

\bibitem{Baudoin2010}
Fabrice Baudoin, Michel Bonnefont, and Nicola Garofalo, \emph{A
  sub-{R}iemannian curvature-dimension inequality, volume doubling property and
  the {P}oincar\'{e} inequality}, Mathematische Annalen (2013), 1--28
  (English).

\bibitem{Baudoin2009}
Fabrice {Baudoin} and Nicola {Garofalo}, \emph{Curvature-dimension inequalities
  and {R}icci lower bounds for sub-{R}iemannian manifolds with transverse
  symmetries}, ArXiv e-prints (2011).

\bibitem{Baudoin2013}
Fabrice Baudoin, Maria Gordina, and Tai Melcher, \emph{Quasi-invariance for
  heat kernel measures on sub-{R}iemannian infinite-dimensional {H}eisenberg
  groups}, Trans. Amer. Math. Soc. \textbf{365} (2013), no.~8, 4313--4350.

\bibitem{BaudoinTeichmann2005}
Fabrice Baudoin and Josef Teichmann, \emph{Hypoellipticity in infinite
  dimensions and an application in interest rate theory}, Ann. Appl. Probab.
  \textbf{15} (2005), no.~3, 1765--1777.

\bibitem{bogachev-gaussian-book}
Vladimir~I. Bogachev, \emph{Gaussian measures}, Mathematical Surveys and
  Monographs, vol.~62, American Mathematical Society, Providence, RI, 1998.

\bibitem{ccgl-positivity-heisenberg}
Chin-Huei Chang, Der-Chen Chang, Peter Greiner, and Hsuan-Pei Lee, \emph{The
  positivity of the heat kernel on {H}eisenberg group}, Anal. Appl. \textbf{In
  press} (2013), --.

\bibitem{Dobbs2013}
Daniel Dobbs and Tai Melcher, \emph{Smoothness of heat kernel measures on
  infinite-dimensional {H}eisenberg-like groups}, J. Funct. Anal. \textbf{264}
  (2013), no.~9, 2206--2223.

\bibitem{Driver1997b}
Bruce~K. Driver, \emph{Integration by parts and quasi-invariance for heat
  kernel measures on loop groups}, J. Funct. Anal. \textbf{149} (1997), no.~2,
  470--547.

\bibitem{Driver2003}
\bysame, \emph{Heat kernels measures and infinite dimensional analysis}, Heat
  kernels and analysis on manifolds, graphs, and metric spaces ({P}aris, 2002),
  Contemp. Math., vol. 338, Amer. Math. Soc., Providence, RI, 2003,
  pp.~101--141. \MR{2039953 (2005e:58058)}

\bibitem{A35}
Bruce~K. Driver and Maria Gordina, \emph{Heat kernel analysis on
  infinite-dimensional {H}eisenberg groups}, J. Funct. Anal. \textbf{255}
  (2008), no.~9, 2395--2461.

\bibitem{A38}
\bysame, \emph{Integrated {H}arnack inequalities on {L}ie groups}, J.
  Differential Geom. \textbf{83} (2009), no.~3, 501--550.

\bibitem{A37}
\bysame, \emph{Square integrable holomorphic functions on infinite-dimensional
  {H}eisenberg type groups}, Probab. Theory Relat. Fields \textbf{147} (2010),
  481--528.

\bibitem{eldredge-precise-estimates}
Nathaniel Eldredge, \emph{Precise estimates for the subelliptic heat kernel on
  {H}-type groups}, J. Math. Pures. Appl. \textbf{92} (2009), 52--85.

\bibitem{fernique70}
Xavier Fernique, \emph{Int\'egrabilit\'e des vecteurs gaussiens}, C. R. Acad.
  Sci. Paris S\'er. A-B \textbf{270} (1970), A1698--A1699.

\bibitem{gaveau77}
Bernard Gaveau, \emph{Principe de moindre action, propagation de la chaleur et
  estim\'ees sous elliptiques sur certains groupes nilpotents}, Acta Math.
  \textbf{139} (1977), no.~1-2, 95--153.

\bibitem{hormander67}
Lars H{\"o}rmander, \emph{Hypoelliptic second order differential equations},
  Acta Math. \textbf{119} (1967), 147--171.

\bibitem{Janson1997}
Svante Janson, \emph{Gaussian {H}ilbert spaces}, Cambridge Tracts in
  Mathematics, vol. 129, Cambridge University Press, Cambridge, 1997.

\bibitem{kuelbs-li-2005}
James Kuelbs and Wenbo Li, \emph{A functional {LIL} for stochastic integrals
  and the {L}\'evy area process}, J. Theoret. Probab. \textbf{18} (2005),
  no.~2, 261--290. \MR{2137043 (2006c:60040)}

\bibitem{Kuo75}
Hui~Hsiung Kuo, \emph{Gaussian measures in {B}anach spaces}, Springer-Verlag,
  Berlin, 1975, Lecture Notes in Mathematics, Vol. 463.

\bibitem{KS1991}
Shigeo Kusuoka and Daniel~W. Stroock, \emph{Precise asymptotics of certain
  {W}iener functionals}, J. Funct. Anal. \textbf{99} (1991), no.~1, 1--74.

\bibitem{ledoux-large-deviations-1990}
Michel Ledoux, \emph{A note on large deviations for {W}iener chaos},
  S\'eminaire de {P}robabilit\'es, {XXIV}, 1988/89, Lecture Notes in Math.,
  vol. 1426, Springer, Berlin, 1990, pp.~1--14.

\bibitem{levy40}
Paul L{\'e}vy, \emph{Le mouvement brownien plan}, Amer. J. Math. \textbf{62}
  (1940), 487--550. \MR{0002734 (2,107g)}

\bibitem{levy50}
\bysame, \emph{Wiener's random function, and other {L}aplacian random
  functions}, Proceedings of the Second Berkeley Symposium on Mathematical
  Statistics and Probability, 1950 (Berkeley and Los Angeles), University of
  California Press, 1951, pp.~171--187.

\bibitem{li-small-ball-Lp}
Wenbo~V. Li, \emph{Small ball probabilities for {G}aussian {M}arkov processes
  under the {$L_p$}-norm}, Stochastic Process. Appl. \textbf{92} (2001), no.~1,
  87--102.

\bibitem{mo-hyperplanes}
George Lowther, \emph{Does infinite-dimensional {B}rownian motion live in
  hyperplanes?}, MathOverflow, URL:http://mathoverflow.net/q/102963 (version:
  2012-07-24).

\bibitem{lust-piquard}
Fran{\c{c}}oise Lust-Piquard, \emph{A simple-minded computation of heat kernels
  on {H}eisenberg groups}, Colloq. Math. \textbf{97} (2003), no.~2, 233--249.

\bibitem{Malliavin1990a}
Paul Malliavin, \emph{Hypoellipticity in infinite dimensions}, Diffusion
  processes and related problems in analysis, Vol.\ I (Evanston, IL, 1989),
  Progr. Probab., vol.~22, Birkh\"auser Boston, Boston, MA, 1990, pp.~17--31.

\bibitem{mansuy-yor}
Roger Mansuy and Marc Yor, \emph{Aspects of {B}rownian motion}, Universitext,
  Springer-Verlag, Berlin, 2008. \MR{2454984 (2010a:60278)}

\bibitem{MattinglyPardoux2006}
Jonathan~C. Mattingly and {\'E}tienne Pardoux, \emph{Malliavin calculus for the
  stochastic 2{D} {N}avier-{S}tokes equation}, Comm. Pure Appl. Math.
  \textbf{59} (2006), no.~12, 1742--1790.

\bibitem{Metivier82}
Michel M{\'e}tivier, \emph{Semimartingales}, de Gruyter Studies in Mathematics,
  vol.~2, Walter de Gruyter \& Co., Berlin, 1982, A course on stochastic
  processes.

\bibitem{randall}
Jennifer Randall, \emph{The heat kernel for generalized {H}eisenberg groups},
  J. Geom. Anal. \textbf{6} (1996), no.~2, 287--316.

\bibitem{vsc-book}
N.~Th. Varopoulos, L.~Saloff-Coste, and T.~Coulhon, \emph{Analysis and geometry
  on groups}, Cambridge Tracts in Mathematics, vol. 100, Cambridge University
  Press, Cambridge, 1992.

\bibitem{yor-filtrations}
Marc Yor, \emph{Les filtrations de certaines martingales du mouvement brownien
  dans {${\bf R}\sp{n}$}}, S\'eminaire de {P}robabilit\'es, {XIII} ({U}niv.
  {S}trasbourg, {S}trasbourg, 1977/78), Lecture Notes in Math., vol. 721,
  Springer, Berlin, 1979, pp.~427--440. \MR{544812 (81b:60045)}

\bibitem{yor-remarques-levy}
\bysame, \emph{Remarques sur une formule de {P}aul {L}\'evy}, Seminar on
  {P}robability, {XIV} ({P}aris, 1978/1979) ({F}rench), Lecture Notes in Math.,
  vol. 784, Springer, Berlin, 1980, pp.~343--346. \MR{580140 (82c:60144)}

\end{thebibliography}
\end{document}